\documentclass[11pt]{amsart}

\RequirePackage[l2tabu, orthodox]{nag}
\usepackage[T1]{fontenc}
\usepackage[utf8]{inputenc}
\usepackage{lmodern}
\usepackage[english]{babel}
\usepackage{color} \usepackage[usenames,dvipsnames]{xcolor}
\usepackage{amsrefs,amsthm,amssymb,amsmath}
\usepackage[pdftex]{graphicx}
\usepackage{enumerate}
\usepackage{mathtools}
\usepackage{caption}
\usepackage{appendix}

\numberwithin{equation}{section}
\numberwithin{figure}{section}

\let\oldtocsection=\tocsection

\let\oldtocsubsection=\tocsubsection

\let\oldtocsubsubsection=\tocsubsubsection

\renewcommand{\tocsection}[2]{\hspace{0em}\oldtocsection{#1}{#2}}
\renewcommand{\tocsubsection}[2]{\hspace{2em}\oldtocsubsection{#1}{#2}}
\renewcommand{\tocsubsubsection}[2]{\hspace{4em}\oldtocsubsubsection{#1}{#2}}

\def\comment#1{}

\newcommand{\R}{\mathbb{R}}
\newcommand{\N}{\mathbb{N}}
\newcommand{\Z}{\mathbb{Z}}
\newcommand{\Sc}{\mathbb{S}^1} 
\newcommand{\T}{\mathbb{S}^1}
\newcommand{\Diff}{\mathsf{Diff}}
\newcommand{\NE}{\mathsf{NE}}

\newcommand{\eps}{\varepsilon}
\newcommand{\PSL}{\mathsf{PSL}}
\newcommand{\PL}{\mathsf{PL}}

\newcommand{\id}{\mathrm{id}} 
\newcommand{\Homeo}{\mathsf{Homeo}}
\DeclareMathOperator{\rot}{\mathsf{rot}}
\DeclareMathOperator{\Isom}{\mathsf{Isom}}
\DeclareMathOperator{\rel}{\mathsf{rel}}
\newcommand{\St}{\mathsf{St}} 

\newcommand{\cI}{\mathcal{I}}
\newcommand{\cJ}{\mathcal{J}}
\newcommand{\cG}{\mathcal{G}}

\newcommand{\tcM}{\widetilde{\mathcal{M}}}

\newtheorem{thm}{Theorem}[section]
\newtheorem{thmA}{Theorem}

\newtheorem{corA}[thmA]{Corollary}
\newtheorem{lem}[thm]{Lemma}
\newtheorem{prop}[thm]{Proposition}

\newtheorem{conj}[thm]{Conjecture}

\newtheorem{claim1}{Claim}

\newtheorem{case}{Case}
\newtheorem{step}{Step}

\theoremstyle{definition}
\newtheorem{dfn}[thm]{Definition}
\theoremstyle{remark}
\newtheorem{rem}[thm]{Remark}
\newtheorem{ex}[thm]{Example}

\usepackage[margin=2.5cm]{geometry}

\usepackage{indentfirst}
\usepackage{microtype}

\usepackage[colorlinks=true,linkcolor=blue,citecolor=magenta]{hyperref}

\title[Ping-pong partitions and locally discrete groups of $\Diff_+^\omega(\T)$, II]{Ping-pong partitions and locally discrete groups of
	real-analytic circle diffeomorphisms, II: Applications}
\author[S. Alvarez, P. G. Barrientos, D. Filimonov, V. Kleptsyn, D. Malicet, C. Meni\~no, M. Triestino]{
	S\'ebastien Alvarez    \and Pablo G.~Barrientos \and Dmitry Filimonov \and Victor Kleptsyn \and
	Dominique Malicet \and Carlos Meni\~{n}o \and Michele Triestino
}
\date{\today}
\begin{document}

\begin{abstract}
	In the first part of this work we have established an efficient method to obtain a topological classification of locally discrete, finitely generated, virtually free subgroups of real-analytic circle diffeomorphisms. In this second part we describe several consequences, among which the solution (within this setting) to an old conjecture by P. R. Dippolito [Ann.~Math.~107 (1978), 403--453] that actions with invariant Cantor sets must be semi-conjugate to piecewise linear actions. In addition, we exhibit examples of locally discrete, minimal actions which are not of Fuchsian type.
	
	\smallskip
	{\noindent\footnotesize \textbf{MSC\textup{2020}:} Primary 37C85, 20E06. Secondary 20E08, 37B05, 37E10.}
\end{abstract}

\maketitle

\section{Introduction and results}

\subsection{On the structure of codimension-one foliations}\label{ss:structure_intro}
In the recent years, probably under the impulse of the monographs by Ghys \cite{ghys-circle} and Navas \cite{Navas2011} there has been an intense activity around the study of groups acting on the circle.
One historical motivation is the \emph{theory of codimension-one foliations}. Indeed, an action of a group $G$ on the circle $\T$ gives rise, by suspension of the action, to a codimension-one foliation of a closed manifold. There is a perfect dictionary between the dynamics of the action on $\T$ and the dynamics of the leaves of the foliation. Foliations defined by suspensions represent a particular class (e.g.~the manifold $M$ must admit a flat circle bundle structure; if the foliation has no compact leaf, then it has a unique minimal set; etc.) however their study is important for developing new techniques, manufacturing examples, and test conjectures.

The main motivation for this work is an old conjecture in foliation theory stated by Dippolito \cite{Dippolito} (see also Question 16 in the recent survey of Navas \cite{Navas-ICM} for the ICM 2018):
\begin{conj}[Dippolito]\label{conj:Dippolito}
	Let $(M,\mathcal F)$ be a codimension-one foliation of a closed manifold which is transversely $C^2$, with exceptional minimal set $\Lambda$. Then there exists a
	transverse measure supported on $\Lambda$ for which the Radon-Nikodym
	derivative of the action of any holonomy pseudogroup is locally constant.
\end{conj}


As Dippolito writes in \cite{Dippolito}, this conjecture is conditioned on the solution of a major open problem:


\begin{conj}[Dippolito]\label{conj:germs}
	Let $(M,\mathcal F)$ be a codimension-one foliation of a closed manifold which is transversely $C^2$, with exceptional minimal set $\Lambda$. Then there exists a semi-exceptional leaf (that is, a boundary leaf of $M\setminus \Lambda$), whose germs of holonomy maps form an infinite cyclic subgroup.
\end{conj}

The only result available in this direction goes back to the Ph.D. thesis of Hector (Strasbourg, 1972), which establishes it under the assumption of non-trivial $r$-jets for elements in the holonomy group, for some $r\ge 1$. This holds for instance in the case of transversely real-analytic foliations. This result is the so-called \emph{Hector's lemma}, see also Navas \cite[Appendix]{Navas2006} for a detailed account. For a short discussion around Conjecture \ref{conj:germs}, see also Question B.1.1 in the list of problems in foliation theory compiled by Langevin \cite{langevin_questions} in 1992.

Our principal contribution is to show (or at least give strong evidences) that indeed Conjecture~\ref{conj:germs} is the unique obstacle towards Conjecture~\ref{conj:Dippolito}. As a consequence of a more general result (Theorem \ref{t:realization theorem}), we are able to establish Conjecture~\ref{conj:Dippolito} for groups of real-analytic diffeomorphisms of the circle:
\begin{thmA}\label{t:Dippolito}\footnote{A little comment about the choice of numbering: the main results of \cite{MarkovPartitions1} appear as Theorems \ref{t:bpvfreeA}, \ref{t:bpvfreeB} and C, so here we start with Theorem \ref{t:Dippolito}, reducing confusion when giving reference.}
	Let $G\subset \Diff^\omega_+(\T)$ be a finitely generated group of real-analytic circle diffeomorphisms acting with a minimal invariant Cantor set. Then the action of $G$ is semi-conjugate to an action by piecewise linear homeomorphisms. More precisely, every such $G$ is semi-conjugate to a subgroup of Thompson's group $T$.
\end{thmA}

The definition of Thompson's group $T$ will be given in Definition \ref{e:Thompson}. Concisely, it is defined as the group of all dyadic piecewise linear homeomorphisms of the circle $\T\cong\R/\Z$.

The proof of Theorem~\ref{t:Dippolito} relies on an old theorem of Ghys \cite{euler} and a recent work of Deroin, Kleptsyn and Navas \cite{DKN2014}, combined with the main result of the companion work by Alonso \textit{et al.}\ \cite{MarkovPartitions1}, after which we have an efficient way to determine a topological classification of groups of real-analytic diffeomorphisms with a minimal invariant Cantor set and we are able to apply Theorem \ref{t:realization theorem}. Precise statements will be recalled in the next section.
We are confident that our strategy for Theorem~\ref{t:Dippolito} can be extended to deal with transversely $C^2$, codimension-one foliations with an exceptional minimal set, with cyclic germs of holonomy of semi-exceptional leaves; however this is not immediate, as one needs to face extra combinatorial problems when passing from groups to pseudogroups.

\subsection{Multiconvergence property}\label{ss:multicon_intro}
The next results give further dynamical information about the actions of virtually free, locally discrete subgroups of $\Diff_+^\omega(\T)$, notably giving a description  of periodic points.

Before stating Theorem~\ref{mthm:multiconvergence}, recall that one of the deepest results on groups acting on the circle is the characterization of  Fuchsian groups (i.e.\ discrete subgroups of $\PSL(2,\R)$), up to $C^0$ conjugacy, by the so-called convergence property (by works of Tukia \cite{Tukia}, Gabai \cite{Gabai}, Casson and Jungreis \cite{Casson-Jungreis}), which is the case $K=1$ in the following definition.

\begin{dfn}\label{d:multiconvergence}
	Let $G\subset \Homeo_+(\T)$ be a subgroup. We say that $G$ has the \emph{multiconvergence property} if there exists a uniform $K\in\N$ such that for every infinite sequence $\{g_n\}_n$ of distinct elements in $G$, there exist finite subsets $A$ and $R\subset \T$, with $\#A,\#R\le K$, and a subsequence $\{g_{n_k}\}_k$ such that the sequence of restrictions
	$\{g_{n_k}\restriction_{\T\setminus R}\}_k$ pointwise converges, as $k\to\infty$, to the locally constant map $g_\infty$, with image $g_\infty(\T\setminus R)=A$, with $\#R$ discontinuity points at $R$ and such that $g_\infty(a)=a$ for every $a\in A\setminus R$.
\end{dfn} 

Typical examples of groups with the multiconvergence property are discrete subgroups of finite central extensions of $\PSL(2,\R)$, but we also have:

\begin{thmA}\label{mthm:multiconvergence}
	If $G\subset \Diff_+^\omega(\T)$ is a finitely generated, virtually free, locally discrete subgroup, then it has the multiconvergence property.
\end{thmA}

This immediately implies the following.

\begin{corA}\label{c:bound_period}
	If $G\subset \Diff_+^\omega(\T)$ is a finitely generated, virtually free, locally discrete subgroup, then the number of fixed points of a non-trivial element in $G$ is uniformly bounded.
\end{corA}

Let us discuss some direct consequences on the possible values of rotation numbers. The following definition is inspired by the monograph by Kim, Koberda and Mj \cite{KKmj}:
\begin{dfn}
	Let $G\subset \Homeo_+(\T)$ be a subgroup. The \emph{rotation spectrum} $\rot(G):=\{\rot(g)\mid g \in G\}$ of $G$ is the collection of rotation numbers of elements of $G$.
\end{dfn}

The classical Denjoy theorem states that if a diffeomorphism $f$ in $\Diff^2_+(\T)$ has irrational rotation number (thus the subgroup $\langle f\rangle\subset \Homeo_+(\T)$ has infinite rotation spectrum), then the action of $\langle f \rangle$ on $\T$ is minimal.
The following corollary is an analogue for finitely generated groups of real-analytic diffeomorphisms:
\begin{corA}\label{cor:infinite_rot_spectrum}
	Let $G$ be a finitely generated subgroup of $\Diff_+^\omega(\T)$ with infinite rotation spectrum. Then $G$ acts minimally on $\T$.
\end{corA}

\begin{proof}
	Suppose $G$ does not act minimally, then either it acts with a finite orbit (in which case the proof ends easily), or it acts with a minimal invariant Cantor set. In the latter case, Denjoy theorem implies that every element has rational rotation number. Moreover by a result of Ghys \cite{euler}, $G$ must be virtually free. Observe that a finitely generated virtually free group cannot have elements of arbitrary large period; also, by Corollary \ref{c:bound_period}, we have a uniform bound on periods of periodic orbits and thus on denominators of rotation numbers.
\end{proof}

\begin{rem}\label{r.Thompson_spectrum} The regularity assumption $G\subset \Diff_+^\omega(\T)$ in Corollary \ref{cor:infinite_rot_spectrum} is necessary: Thompson's group $T\subset \mathsf{Homeo}_+(\T)$ has infinite rotation spectrum (it contains all dyadic rotations, and actually $\rot(T)=\mathbb Q/\mathbb Z$, see Lemma \ref{e:Thompson}) and is semi-conjugate to a subgroup of $\Diff^\infty_+(\T)$ acting with a minimal invariant Cantor set (Ghys and Sergiescu \cite{ghys-sergiescu}).
\end{rem}

In \cite{Matsuda2009}, Matsuda proved that if a subgroup $G\subset \Diff_+^\omega(\T)$ is non-locally discrete and acts without finite orbits, then $G$ has infinite rotation spectrum. As a direct consequence of Corollary \ref{c:bound_period}, we obtain that the converse holds for virtually free groups:
\begin{corA}\label{cor:rot_spectr_charac}
	Let $G$ be a finitely generated, virtually free subgroup of $\Diff^\omega_+(\T)$, whose action on $\T$ has no finite orbit. The following statements are equivalent:
	\begin{enumerate}
		\item $G$ is locally discrete;
		\item the rotation spectrum $\rot(G)=\{\rot(g)\mid g \in G\}$ is finite. 
	\end{enumerate}
\end{corA}
\begin{rem} After Conjecture \ref{c-missing-piece}, we expect that the statement of the corollary remains valid without the virtual-freeness assumption.
\end{rem}

\begin{rem}
	Let us mention an open problem suggested by Matsuda \cite{Matsuda2009}: if $G\subset \Diff_+^\omega(\T)$ has infinite rotation spectrum, does it contain an element with irrational rotation number?
\end{rem}

\subsection{On the classification of locally discrete groups}\label{ssc:locally_discrete}
For the next result, we recall some context, starting from the following definition.

\begin{dfn}\label{d:locally_discrete}
	A subgroup $G\subset \Diff_+^\omega(\T)$ is ($C^1$-)\emph{locally discrete} if for every interval $I\subset \T$, 
	the restriction of the identity to $I$ is isolated in the $C^1$ topology among the set of restrictions 
	\[G\restriction_I=\{g\restriction_I\}_{g\in G}\subset C^1(I;\T).\]
\end{dfn}

\begin{rem}\label{r:discrete}
	Notice that \emph{a priori} the notion of local discreteness is stronger than simple discreteness. However no example of discrete, non-locally discrete group is known. In this direction, Bertrand Deroin informed us that he is able to find many such examples inside the group of $C^2$ diffeomorphisms. However, the question in real-analytic regularity stays open, see also the discussion about work of Eskif and Rebelo \cite{ER2015} below.
\end{rem}

Also, when $G\subset \Diff_+^1(\T)$ is a group of $C^1$ diffeomorphisms, we introduce the subset of \emph{non-expandable points} $\NE(G)=\{x\in \T\mid g'(x)\le 1\,\text{for every }g\in G\}$.
Recent developments of Rebelo, the authors and collaborators \cite{rebelo,FKone,DKN2014,football,Deroin} have clarified the picture for locally discrete groups:

\begin{thm}\label{t-state-of-the-art}Let $G\subset \Diff_+^\omega(\T)$  be a finitely generated subgroup.
	\begin{enumerate}[\rm (1)]
		\item If $\NE(G)\neq\varnothing$, then $G$ is locally discrete.
		\item If $G$ is locally discrete and is 
		\begin{itemize}
			\item neither one-ended, non-finitely presentable, nor
			\item one-ended, with unbounded torsion,
		\end{itemize} 
		then $G$ must be of the following type:
		\begin{itemize}
			\item either virtually free and $\NE(G)\neq\varnothing$, or
			\item $C^\omega$ conjugate to a finite central extension of a cocompact Fuchsian group, and $\NE(G)=\varnothing$.
		\end{itemize}
	\end{enumerate}
\end{thm}

\begin{conj}[``Missing Piece'' Conjecture from Alvarez \textit{et al.} \cite{football}]\label{c-missing-piece}
	The dichotomy in Theorem \ref{t-state-of-the-art} holds for every locally discrete, finitely generated subgroup $G\subset \Diff_+^\omega(\T)$.
\end{conj}
This suggests that locally discrete subgroups of $\Diff_+^\omega(\T)$ should be quite elementary, and that subgroups of $\Diff_+^\omega(\T)$ of more complicated type should be non-locally discrete. For instance, $\PSL(2,\R)\subset \Diff_+^\omega(\T)$ contains some closed hyperbolic 3-manifold groups as non-locally discrete subgroups (this is well-known to experts and one early reference is in the book of Maclachlan and Reid \cite[Section 13.7]{MR2003book}; see also the discussion in Bonatti, Kim, Koberda and Triestino \cite{BKKT}).

Natural examples of locally discrete groups are the fundamental groups of hyperbolic surfaces of finite type (they are discrete in $\PSL(2,\R)\subset\Diff^\omega_+(\T)$) and their finite central extensions. For a certain time, this was the only known class of such examples, thus motivating Ghys to ask whether any locally discrete subgroup of $\Diff^\omega_+(\T)$ acting minimally on the circle comes from a surface group (personal communication). We will see in this work (Section \ref{s:exotic}) that ping-pong partitions give a simple combinatorial way to construct a rich collection of exotic examples of locally discrete subgroups $G\subset \Diff_+^\omega(\T)$, giving negative answer to the question of Ghys. The main technical ingredients are Theorem \ref{t:realization theorem} (realization in $\Diff_+^\omega(\T)$), and Theorem \ref{t.minimal_actions} for realizing \emph{minimal} free group actions.

On the other hand, finitely generated subgroups of $\Diff^\omega_+(\T)$ that are non-locally discrete have been extensively studied by different authors in the last two decades. In particular, they display strong rigidity properties. As a sample statement, let us mention the following result by Eskif and Rebelo \cite{ER2015}, which is about groups which are non-$C^2$-locally discrete (the definition is the same as above, but considering the $C^2$ closure):
\begin{thm}[Eskif and Rebelo]\label{t:ERthmA}
	Consider two finitely generated, non-abelian subgroups $G_1$ and $G_2$
	of $\Diff^\omega_+(\T)$. Suppose that these groups are non-$C^2$-locally discrete. Then every homeomorphism $h:\T\to \T$ satisfying $G_2= hG_1h^{-1}$ coincides with an element of $\Diff^\omega_+(\T)$.
\end{thm}

Clearly if $G$ is non-$C^2$-locally discrete, then it is also non-$C^1$-locally discrete, but the converse is currently unknown (cf.\ Remark  \ref{r:discrete}).
However, we have the following statement, which extends \cite[Theorem B]{ER2015} by Eskif and Rebelo (by a completely different approach, we are able to remove a technical assumption).

\begin{corA}
	Suppose that $\Gamma$ is a finitely generated, non-elementary Gromov-hyperbolic group, and consider two topologically conjugate faithful representations $\rho_1,\rho_2: \Gamma\to\Diff_+^\omega(\T)$. Assume that $ \rho_1(\Gamma)\subset \Diff_+^\omega(\T)$ is non-$C^2$-locally discrete. Then every orientation-preserving homeomorphism $h:\T\to \T$ conjugating the representations $\rho_1$ and $\rho_2$ coincides with an element of $\Diff^\omega_+(\T)$.
\end{corA}

\begin{proof}
	If $G_1=\rho_1(\Gamma)$ (and thus $G_2=\rho_2(\Gamma)$) admits a finite orbit, then we can repeat the quite standard argument in \cite{ER2015}. Otherwise, after \cite[Theorem 5.1]{ER2015} (which is an intermediate step for Theorem \ref{t:ERthmA}), it is enough to prove that $\NE(G_2)=\varnothing$. Recall that a Gromov-hyperbolic group is always finitely presented and with bounded torsion \cite[Chapter III.$\Gamma$]{bridson-haefliger}.
	After Theorem~\ref{t-state-of-the-art}, if $\NE(G_2)\neq\varnothing$, the Gromov-hyperbolic group $G_2$ is virtually free and therefore it admits a ping-pong partition by Theorem \ref{t:bpvfreeA}.	
	Therefore $G_1$, being $C^0$ conjugate to $G_2$, also admits a ping-pong partition. This contradicts that $G_1$ is non-$C^2$-locally discrete.
\end{proof}

\subsection{Organization of the paper}

The whole work is based upon the notion of ping-pong partitions for virtually free group actions on the circle.
We start by recalling in Section \ref{s:previous} the setting and the results from the companion work \cite{MarkovPartitions1}. In Section \ref{s:realization} we explain how to realize ping-pong partitions by actions with desired properties, which notably gives the proof of Theorem \ref{t:Dippolito} (realizations by dyadic piecewise linear homemorphisms) and several interesting examples of locally discrete groups of real-analytic diffeomorphisms with minimal actions. In Section \ref{s:dynamics} we study extensively the dynamical properties obtained from the ping-pong partitions defined after the work of Deroin, Kleptsyn, and Navas \cite{DKN2014}, which leads in particular to the proof of the multiconvergence property (Theorem \ref{mthm:multiconvergence}). The text is completed by Appendix \ref{appendix}, where we compare the original construction by Deroin, Kleptsyn and Navas appearing in \cite{DKN2014}, with the construction revisited in the companion work \cite{MarkovPartitions1}. 

\section{Ping-pong partitions for virtually free groups}
\label{s:previous}

As we explained in Section \ref{ssc:locally_discrete}, an important aspect of the classification of locally discrete subgroups $G\subset \Diff_+^\omega(\T)$ is to understand how \emph{virtually free groups} may act. For this, we have introduced
in \cite{MarkovPartitions1} the notion of \emph{ping-pong partition}, which is a semi-conjugacy invariant, depending on a \emph{marking}. We briefly recall the notions appearing in \cite{MarkovPartitions1}. It is a classical result that a finitely generated group is virtually free if and only if it admits a proper isometric action (i.e.\ with finite stabilizers) on a locally finite simplicial tree, with bounded fundamental domain. A \emph{marking} of a virtually free group $G$ is the choice of an orientation-preserving isometric action $\alpha:G\to\Isom_+(X)$ of $G$ on a tree $X$ and a fundamental domain $T\subset X$. Given such a marking, we denote by $\overline X=(V,E)$ the quotient graph of the action, and the choice of $T\subset X$ determines a spanning tree in $\overline X$, with set of oriented edges $E_T\subset E$. We write $S:=E\setminus E_T$. The marking determines finite vertex groups $G_v\subset G$ (for $v\in V$) and edge groups $A_e$ (for $e\in E$) with boundary injections $\alpha_e:A_e\to G_{o(e)}$, $\omega_e:A_e\to G_{t(e)}$ (here $o(e)$ and $t(e)$ denote respectively the origin and target of the oriented edge $e\in E$). This gives the following presentation of the group $G$, as the fundamental group of a graph of groups:
\begin{equation}\label{eq:pres_fundamental_group}
	G\cong \pi_1(\overline X;G_v,A_e)=\sbox0{$G_v,E\, \left\vert \begin{array}{lr}
			\rel (G_v)& \text{for }v\in V,\\
			\overline{e}=e^{-1}&\text{for }e\in E,\\
			e=\id & \text{for }e\in E_T,\\
			e^{-1}\alpha_e(g)e=\omega_e(g)&\text{ for }e\in E, g\in A_e\end{array}\right.$}
	\mathopen{\resizebox{1.3\width}{\ht0}{$\Bigg\langle$}}
	\usebox{0}
	\mathclose{\resizebox{1.3\width}{\ht0}{$\Bigg\rangle$}}.
\end{equation}

\begin{rem}\label{r:mthmC}
	Virtually free groups of circle homeomorphisms are of very special form: vertex groups are finite cyclic, and in fact they are always free-by-finite cyclic (see \cite[Theorem C]{MarkovPartitions1}).
\end{rem}

\begin{dfn}[Interactive family]\label{d:generalized_ping-pong}
	Consider a graph $\overline X=(V,E)$ and let $G=\pi_1(\overline X;G_v,A_e)$ be the fundamental group of a graph of groups.
	Choose a spanning tree $T=(V,E_T)\subset \overline X$, and let $S=E\setminus E_T$ be the collection of oriented edges not in $T$. We denote by $\mathsf{St}_T(v)=\{e\in E_T\mid o(e)=v\}$ the \emph{star} of $v$ in $T$;  given $e\in E_T$ and $v\in V$, we also write $v\in C(e,T)$ if the edge $e$ belongs to the oriented geodesic path in $T$ connecting $o(e)$ to $v$.
	
	Given an action of $G$ on a set $\Omega$, a family of subsets $\boxminus=\{X_v,Z_s\}_{v\in V,s\in S}$ is called an \emph{interactive family} if:
	\begin{enumerate}[(\text{IF} 1)]
		
		\item $X_v$, $Z_s$ (for $v\in V,s\in S$) are pairwise disjoint subsets; the $Z_s$ are non-empty, and if $G_v\neq \alpha_e(A_e)$ for some $e\in E$ such that $o(e)=v$, then $X_v\neq \varnothing$;
		\label{pp1}
		
		\item for every $s\in S$ and $O\in\boxminus\setminus \{Z_{\overline s}\}$, one has $s(O)\subset Z_s$; \label{pp2}
		
		\item  $\alpha_s(A_s)(Z_s)\subset Z_s$ for $s\in S$;\label{pp3}
		
		\item $(G_{o(s)}\setminus \alpha_s(A_s))(Z_s)\subset X_{o(s)}$ for $s\in S$;\label{pp4}
		
		\item for $v\in V$ and $e\in \mathsf{St}_T(v)$ such that $G_v\neq \alpha_e(A_e)$, $X_v$ contains a non-empty $X_v^e$;\label{pp5}
		
		\item $\alpha_e(A_{e})(X_{o(e)}^e)\subset X^e_{o(e)}$ for $e\in E_T$; \label{pp6}
		
		\item if $e\in E_T$, $v\in V$ are such that $v\in C(e,T)$, then $\left (G_{o(e)}\setminus \alpha_e(A_e)\right )(X_{v})\subset X^e_{o(e)}$ (in particular this holds for $v=t(e)$);\label{pp7} 
		
		\item if $e\in E_T, s\in S$ are such that $o(s)\in C(e,T)$, then $\left (G_{o(e)}\setminus \alpha_e(A_e)\right )(Z_s)\subset X^e_{o(e)}$.\label{pp8}
		
	\end{enumerate}
	
	In addition, one says that the interactive family is \emph{proper} if the following holds:
	
	\begin{enumerate}[(\text{IF} 1)]
		\setcounter{enumi}{8}
		
		\item for $s\in S$, the restriction of the action of $\alpha_s(A_s)$ to $Z_s$ is faithful, and similarly, for $e\in E_T$,  the restriction of the action of $\alpha_e(A_e)$ to $X_{o(e)}^e$ is faithful;\label{pp9}
		
		\item if there exists a vertex $v\in V$ such that $X_v\neq\varnothing$, then there exists a (possibly different) vertex $w\in V$ such that the union of all the images from (IF~\ref{pp4},\ref{pp7},\ref{pp8}) inside the corresponding $X_{w}$ misses a point;
		\label{pp10}
		
		\item if $S=\{s,\overline s\}$ and $X_v=\varnothing$ for every $v\in V$, then we require that there exists a point $p\in \Omega\setminus(Z_s\cup Z_{\overline{s}})$ such that $s(p)\in Z_s$ and $\overline s(p)\in Z_{\overline s}$.\label{pp11}
	\end{enumerate}
\end{dfn}

\begin{dfn}\label{d:gaps}
	Let  $\mathcal I$ be a collection of finitely many disjoint open intervals of the circle $\T$. A \emph{gap} of $\mathcal I$ is a connected component of the complement of $\bigcup_{I\in \mathcal I}I$ in $\T$. We denote by $\mathcal J$ the collection of gaps of the partition $\mathcal I$. 
	
	Given a homeomorphism $g:\T\to\T$ and $I\in \mathcal I$, we say that the image $g(I)$ is \emph{$\mathcal{I}$-Markovian} if it coincides with a union of intervals $I_0,\ldots,I_m\in \mathcal I$ and gaps $J_1,\ldots,J_m\in\mathcal J$ (where $m\ge 1$). We will also informally write that $g$ \emph{expands} the interval $I$.
\end{dfn}

\begin{dfn}[Ping-pong partition]\label{d:markov-partition}
	Let $G\subset \Homeo_+(\T)$ be a virtually free group and $(\alpha:G\to\Isom_+(X), T)$ a marking. Let $\cG=\{G_v,\alpha_s(A_s)s\}_{v\in V,s\in S}$ be the preferred system of generators for $G$.
	
	A \emph{ping-pong partition} for $(G,\alpha,T)$ is a collection $\Theta=\{U_v^e,V_s\}_{v\in V, e\in\St_{\overline X}(v), s\in S}$ of open subsets of the circle $\T$, satisfying the following properties.
	\begin{enumerate}[(PPP 1)]
		\item Letting $U_v=\bigcup_{e\in \St_{\overline X}(v)}U_v^e$, the family $\{U_v,V_s\}_{v\in V, s\in S}$ defines an interactive family in the sense of Definition~\ref{d:generalized_ping-pong}, with three additional requirements: for any $v\in V$
		\begin{itemize}
			\item the subsets $U_v^e$, for $e\in\St_T(v)$, are the subsets required for (IF \ref{pp5}),
			\item the subsets $U_v^s$, for $s\in S$ such that $o(s)=v$, are such that $\left (G_v\setminus \alpha_s(A_s)\right )\left (V_s\right )\subset U_v^s$, strengthening (IF \ref{pp4}), and moreover $\alpha_s(A_s)(U_v^s)=U_v^s$,
			\item the subsets $U_v^e$, for  $e\in\St_{\overline X}(v)$, are pairwise disjoint.
		\end{itemize}\label{ppp1}
		\item Every atom of $\Theta$ is the union of finitely many intervals.\label{ppp2}
		\item For every element $g\in \cG$ and every connected component $I$ of some $O\in\Theta$
		\begin{itemize}
			\item either there exists $O'\in\Theta$ such that $g(I)\subset O'$, or
			\item the image $g(I)$ is $\cI$-Markovian, where $\cI$ is the collection of connected components of elements of $\Theta$.
		\end{itemize}\label{ppp3}

	\end{enumerate}	
	In addition, if $\Theta$ defines a proper interactive family, we will say that  the ping-pong partition is \emph{proper}.	
\end{dfn}

\begin{dfn}[Equivalence of partitions]\label{d:equivalence_PPP}
	Let $(G,\alpha,T)$ be a marked virtually free group, with preferred generating set $\cG=\{G_v,\alpha_s(A_s)s\}_{v\in V,s\in S}$ and let $\rho,\rho':G\to\Homeo_+(\T)$ be two representations having interactive families $\Theta=\{U^e_v,V_s\}$, $\Theta'=\{{U^e_v}',V'_s\}$ respectively. Denote by $\cI,\cI'$ the corresponding sets of connected components. A map $\theta:\cI\to\cI'$ is a \emph{ping-pong equivalence} if the following conditions are satisfied:
	\begin{enumerate}[(PPE 1)]
		\item $\theta$ is a bijection which preserves the cyclic ordering of the intervals;\label{ppe1}
		\item $\theta$ preserves the inclusions relations of the two ping-pong partitions:
		\begin{itemize}
			\item if  $g\in \cG,I_1, I_2\in \cI$ are such that $\rho(g)(I_1)\subset I_2$, then $\rho'(g)(\theta(I_1))\subset \theta(I_2)$, and, in case of equality, equality is preserved,
			\item whereas if $g\in \cG$ and $I\in \cI$ are such that $\rho(g)(I)$ is $\mathcal I$-Markovian, union of intervals $I_0,\ldots,I_m\in\cI$ and gaps $J_1,\ldots, J_m\in\cJ$, then $\rho'(g)(\theta(I))$ is $\cI'$-Markovian, union of the intervals $\theta(I_0),\ldots,\theta(I_m)\in\cI'$ and gaps $\theta(J_1),\ldots, \theta(J_m)\in\cJ$.
		\end{itemize} \label{ppe3}
	\end{enumerate}
\end{dfn}

\begin{dfn}[Semi-conjugacy]
	Let $\rho,\rho':G\to \Homeo_+(\T)$ be two representations. They are \emph{semi-conjugate} if the following holds: there exist
	\begin{itemize}
		\item a monotone non-decreasing map $h:\R\to \R$ commuting with the integer translations and
		\item two corresponding central lifts $\widehat{\rho},\widehat \rho':\widehat{G}\to \Homeo_{\Z}(\R)$  to homeomorphisms of the real line commuting with integer translations,
	\end{itemize}
	such that
	\[
	h\,\widehat{\rho}(\widehat{g})=\widehat{\rho}'(\widehat{g})\,h,\quad\text{for any }\widehat{g}\in \widehat{G}.
	\]
\end{dfn}

We can now recall the main results of \cite{MarkovPartitions1}.

\setcounter{thmA}{0}
\begin{thmA}\label{t:bpvfreeA}
	Let $G\subset \Diff^\omega_+(\T)$ be a locally discrete, virtually free group of real-analytic circle diffeomorphisms. For any marking $(\alpha:G\to \Isom_+(X),T)$, there  exists a proper  ping-pong partition for the action of $G$ on $\T$ (in the sense of Definition~\ref{d:markov-partition}).
\end{thmA}

\begin{thmA}\label{t:bpvfreeB}
	Let $\rho,\rho':(G,\alpha,T)\to \Homeo_+(\T)$ be two  representations of a virtually free group with a marked action $\alpha$ on a tree. Suppose that the actions on $\T$ have equivalent proper ping-pong partitions (in the sense of Definition~\ref{d:equivalence_PPP}). Then the actions are semi-conjugate.
\end{thmA}

\section{Realization}\label{s:realization}

In this section we discuss the problem of realization. That is, given a ping-pong partition for a marked virtually free group $(G,\alpha,T)$ of circle homeomorphisms, is it possible to find a new action of $(G,\alpha,T)$ with equivalent ping-pong partition, but of a given desired regularity? We will be particularly interested in the case of piecewise linear regularity (Theorem \ref{t:Dippolito}, related to Dippolito's conjecture) and real-analytic regularity (for the examples in Section \ref{s:exotic}).

Some of the results are obtained by performing small perturbations of the original action. For this, we recall that the group $\Homeo_+(\T)$ is a topological group when endowed with the topology of uniform convergence. This topology is determined by the uniform distance: given $f$ and $g$ in $\Homeo_+(\T)$, define $d_{C^0}(f,g)=\sup_{x\in \T}d_{\T}(f(x),g(x))$, where 
$d_{\T}$ denotes the distance on the circle on $\T\cong \R/\Z$ induced by the Euclidean distance on the line (so $d_{\T}(x,y)$ is the length of the shortest connected component of $\T\setminus \{x,y\}$).
This also defines a topology on the space of representations of a group $G$ in $\Homeo_+(\T)$: two representations $\rho_1$ and $\rho_2$ in $\mathsf{Hom}(G,\Homeo_+(\T))$ are $\eps$-close if there exists a symmetric generating system $\cG$ of $G$ such that $\sup_{s\in \cG}d_{C^0}(\rho_1(s),\rho_2(s))\le \eps$.

When two points $x,y\in \T$ are sufficiently close, we can use a Euclidean chart to measure their distance, and we will often write $|x-y|=d_{\T}(x,y)$, where $x-y$ denotes the usual operation in any Euclidean chart.

\subsection{Combinatorial realization}

\begin{dfn}\label{d.combinatorial}
	Let $D_0\subset \T$ be a dense subset. A subgroup $G_0\subset \Homeo_+(\T)$ is \emph{$D_0$-combinatorial} if it satisfies the following conditions.
	\begin{enumerate}[(C 1)]
		\item \label{invariance} $D_0$ is $G_0$-invariant.
		
		\item \label{interpolation} For any $n\in\N$ and any two $n$-tuples
		$(x_1,\ldots,x_n)$ and $(y_1,\ldots,y_n)$ of circularly ordered
		points of $D_0$, there exists an element of $G_0$ such
		that $g(x_i)=y_i$ for all $i\in\{1,\ldots,n\}$.
		\item \label{order} For any $q\in \mathbb{N}$, the subgroup $G_0$ contains an element of order $q$.
		\item \label{conjugation} For any two elements $a$ and $b\in G_0$ of finite order $q\in\N$ and same rotation number,
		and for any finite subset $E\subset D_0$ such that $a^k\restriction_E=b^k\restriction_E$ for
		every $k\in\{1 ,\ldots,q-1\}$, there exists an element $h \in G_0$ such that $b=hah^{-1}$ and
		$h\restriction_E=\mathrm{id}\restriction_E$.
		
	\end{enumerate}
	When $D_0=\T$ we simply say that $G_0$ is \emph{combinatorial}.
\end{dfn}

\begin{ex}\label{e:piecewiseaffine} Let $D_0=\T$ and let $G_0=\PL_+(\T)$ be the group of  orientation-preserving,
	piecewise linear homeomorphisms of the circle. We claim that $\PL_+(\T)$ is combinatorial.
	Properties~(C \ref{invariance}--\ref{order})
	are clearly satisfied. To establish condition~(C \ref{conjugation}), let $a$ and $b$ be as in Definition \ref{d.combinatorial}. Up to replace $a$ and $b$ by appropriate powers, we can assume $\rot(a)=\rot(b)=1/q$.
	Fix a point $x_0\in E$ and consider the 
	interval $I_0:=[x_0,a(x_0))=[x_0,b(x_0))\subset \T$, which is a fundamental domain for both $a$ and $b$. We have that $a^k(I)=b^k(I)=:I_k$ for every $k\in\{0,\ldots, q-1\}$.
	Consider the map $h:\T\to \T$, defined by
	\[
	h(x):=b^ka^{-k}(x)\quad\text{for }x\in I_k.
	\]
	It follows that $ha=bh$ and, by assumption, we have $h(x)=x$ for every $x\in E$. In particular, $h$ preserves every interval $I_k$, on which it is defined by the piecewise linear map $b^ka^{-k}$, so $h\in \PL_+(\T)$.
	This proves that the group of orientation-preserving,  piecewise linear  homeomorphisms of the circle is combinatorial.
\end{ex}

We need to slightly improve the previous example, passing from the uncountable group $\PL_+(\T)$ to a countable (even finitely presented!) subgroup.

\begin{dfn}\label{d:Thompson}
	A \emph{dyadic rational} (or \emph{dyadic rational number}) is a point of $\Z[\tfrac12]/\Z\subset \R/\Z\cong \T$,  that is, a point of the form $p/2^q\pmod \Z$, where $p,q\in\N$. 		
	\emph{Thompson's group $T$} is the group of all \emph{dyadic} piecewise linear homemomorphisms, that is, homeomorphisms of the circle $\R/\Z\cong \T$ which are locally of the form $x\mapsto 2^kx+p/2^q$, for $k,p\in \Z$ and $q\in\N$, with finitely many discontinuity points for the derivative, all at dyadic rationals.
\end{dfn}

See \cite{CFP,BieriStrebel} as standard references on Thompson's groups.

\begin{lem}\label{e:Thompson}
	Thompson's group $T$ is $\Z[\tfrac12]/\Z$-combinatorial.
\end{lem}
\begin{proof}
	Condition~(C \ref{invariance}) follows from the definition of $T$, condition~(C \ref{interpolation}) is classical (see for instance \cite[\S A5.3]{BieriStrebel}), and so is condition~(C \ref{order}) (see \cite[Proposition III.2.1]{ghys-sergiescu}). Finally, the same argument provided in Example \ref{e:piecewiseaffine} shows condition~(C \ref{conjugation}).
	
	We detail a different proof of condition~(C \ref{interpolation}). Given $n$-tuples $(x_1,\ldots,x_n)$ and $(y_1,\ldots,y_n)$ as in (C \ref{interpolation}), we just need to find a dyadic piecewise linear homeomorphism $f_i:[x_i,x_{i+1}]\to [y_i,y_{i+1}]$ for every $i\in \{1,\ldots,n\}$ (set $x_{n+1}=x_1$ and $y_{n+1}=y_1$); we will then consider the dyadic piecewise linear map $g$ such that $g\restriction_{[x_i,x_{i+1}]}=f_i$ for every $i\in \{1,\ldots,n\}$.
	For this, it suffices to exhibit partitions $[x_i,x_{i+1}]=I_1\cup \dots \cup I_r$ and $[y_i,y_{i+1}]=J_1\cup \dots \cup
	J_r$ by dyadic intervals (that is, with dyadic rational endpoints), so that $|J_j| / |I_j|$ is a power of $2$ for every $j\in \{1,\ldots,r\}$; we will then declare $f_i$ to be the unique piecewise linear homeomorphism which maps $I_j$ to $J_j$ linearly for all $j\in \{1,\ldots,k\}$. 
	
	Note that there exist $m,k \in \mathbb{N}$ and $\ell\in\mathbb{Z}$ such that $|[x_i,x_{i+1}]|=\frac{m}{2^\ell}$ and $|[y_i,y_{i+1}]|=\frac{k}{2^\ell}$. Without loss of generality (considering the inverse map if necessary), we can assume $m\leq k$. 
	Choose $r=k$, let $J_1,\ldots,J_k$ be the partition of $[y_i,y_{i+1}]$ by $k$ subintervals of even size $\frac{1}{2^\ell}$ and let $I_1,\ldots, I_{k}$ be consecutive subintervals of $[x_i,x_{i+1}]$ determined by the conditions:
	\[
	|I_j|=\left\{
	\begin{array}{lr}
		\dfrac{1}{2^\ell} &  \text{for }j\in \{1,\ldots,m-1\},\\[1em]
		\dfrac{1}{2^{\ell+j-(m-1)}} & \text{for }j\in \{m,\ldots,k-1\},\\[1em]
		\dfrac{1}{2^{\ell+k-m}} & \text{for }j=k.
	\end{array}
	\right.
	\]
	This defines the map $f_i$ and hence the desired element $g\in T$.
	%
	%
It follows that Thompson's group $T$ is $\Z[\tfrac12]/\Z$-combinatorial.
\end{proof}

The following result is quite standard (see e.g.\ \cite[Corollary A5.8]{BieriStrebel} for a proof for many groups of piecewise linear homeomorphisms).

\begin{lem}\label{l.combinatorial_dense}
Let $D_0\subset \T$ be a dense subset and let $G_0\subset\Homeo_+(\T)$ be a $D_0$-combinatorial group of circle homeomorphisms. Then $G_0$ is $C^0$-dense in $\Homeo_+(\T)$.
\end{lem}
\begin{proof}Fix $s\in \Homeo_+(\T)$.
Given  sufficiently small $\eps>0$, choose a finite subset $B\subset D_0$ such that for every two consecutive points $b_1$ and $b_2$ in $B$, one has $|s(b_1)-s(b_2)|\le \eps/2$.
Then, for every  $b\in B$, choose a point $b'\in D_0$ such that $|s(b)-b'|\le \eps/2$, and such that the assignment $s(b)\mapsto b'$ 
is a circular order preserving bijection.
We claim that any element $g\in G_0$ such that $g(b)=b'$ for every $b\in B$ (which exists by condition (C \ref{interpolation})), satisfies $\sup_{x\in \T}|g(x)-s(x)|\le \eps$.
Indeed, if $x=b\in B$, then $|s(x)-g(x)|=|s(b)-b'|\le \eps/2$ by construction; otherwise
given $x\in \T\setminus B$, denote by $[b_1,b_2]\subset \T$ the smallest interval containing $x$, whose endpoints are in $B$. Then, we have
\[
s(x)-g(x)\le s(b_2)-g(b_1)\le s(b_2)-s(b_1)+\eps/2\le \eps,
\]
and similarly
\[
s(x)-g(x)\ge s(b_1)-g(b_2)\ge s(b_1)-s(b_2)-\eps/2\ge -\eps,
\]
as desired.
\end{proof}

The following result improves condition (C \ref{conjugation}).

\begin{lem}\label{l.C4}
Let $D_0\subset \T$ be a dense subset and let $G_0\subset\Homeo_+(\T)$ be a $D_0$-combinatorial group of circle homeomorphisms.
For every $\eps>0$ there exists $\delta$ such that if $a$ and $b\in G_0$ have finite order $q\in \N$ and same rotation number,  with $d_{C^0}(a,b)<\delta$, then for every finite subset $E\subset D_0$ such that $a^k\restriction_E=b^k\restriction_E$ for
every $k\in\{1 ,\ldots,q-1\}$, then there exists an element $h\in G_0$ as in condition (C~\ref{conjugation}), which moreover satisfies $d_{C^0}(h,\id)<\eps$.
\end{lem}

\begin{proof}
Fix $\eps>0$, and let $a$ and $b$ be two torsion elements of same rotation number $\rot(a)=\rot(b)=1/q$, with $\sup_{i\in \{1,\ldots,q-1\}}\left \{d_{C^0}(a^i,b^i)\right \}<\delta$, where the constant $\delta>0$ will be fixed later (note that as $\Homeo_+(\T)$ is a topological group, it is not restrictive, taking a smaller $\delta$, to assume that all powers of $a$ and $b$ are close).

Take a finite subset $E\subset G_0$ as in the statement. Note that if $E=\varnothing$, then the result is  simple consequence of Lemma \ref{l.combinatorial_dense}, as it holds for $D_0=\T$ and $G_0=\Homeo_+(\T)$.
Fix a point $p\in E$, and consider the interval $I=[p,a(p)]=[p,b(p)]$, which is a fundamental domain for both $a$ and $b$.
Take a finite $a$-invariant subset $B\subset D_0$ which is $\delta$-dense in $\T$ (here we use condition (C \ref{invariance})). Set $B_0=B\cap I$. Then, by condition (C \ref{conjugation}), we can find an element $k\in G_0$ such that
$ka^i(x)k^{-1}=b^i(x)$ for every $i\in\{0,\ldots,q-1\}$ and $x\in B_0\cup E$. As in the proof of Lemma \ref{l.combinatorial_dense}, this gives $d_{C^0}(k,\id)<2\delta$. Set $\overline a=k ak^{-1}$, which is an element in $G_0$ satisfying $\overline a^i(x)=b^i(x)$ for every $x\in B\cup E$. Therefore, by condition (C \ref{conjugation}), there exists an element $\overline h\in G_0$ such that $\overline h \overline a=b\overline h$ and $\overline h(x)=x$ for every $x\in B\cup E$. In particular $\overline h$ is $\delta$-close to the identity. Taking $\delta$ small enough, we see that the composition $h=\overline hk\in G_0$ satisfies the desired properties.
\end{proof}

We are ready to discuss the main technical result of this section.

\begin{prop}[Combinatorial perturbation]\label{p:realization lemma} Let $D_0\subset \T$ be a dense subset and let $G_0\subset\Homeo_+(\T)$ be a $D_0$-combinatorial group of circle homeomorphisms. Let $(G,\alpha,T) \subset\Homeo_+(\T)$ be a marked finitely generated virtually
free group of circle homeomorphisms, and let $\cG\subset G$ be the system of generators associated with the marking $\alpha:G\to \Isom_+(X)$, as in Definition \ref{d:markov-partition}. Then for every $\eps>0$ and every finite subset $P\subset D_0$ 
such that $\cG^2(P)\subset D_0$,\footnote{Note that $\cG^2$ always contains $\id$, therefore $\cG^2(P)=P\cup\cG(P)\cup \cG^2(P)$.} there exists a homomorphism $\pi:G\to G_0$ such that the following conditions are satisfied.
\begin{enumerate}[\rm (1)]
	\item  For every $s\in \cG$, the homeomorphisms $s$ and $\pi(s)$ have the same order, and the rotation numbers of $s$ and $\pi(s)$ agree when $\rot(s)$ is rational.\label{i:rotation}
	\item For every $s\in \cG$, the image $\pi(s)$ and $s$ are $\eps$-close (with respect to the uniform distance).\label{i:close}
	\item $\pi(s)\restriction_P=s\restriction_P$
	for any $s\in \cG$. Moreover, if $s\in G_v$ is an element of a vertex group, then $\pi(s)\restriction_{G_v(P)}=s\restriction_{G_v(P)}$.\label{i:restriction}	
\end{enumerate} 
\end{prop}

\begin{rem}
Observe that condition \eqref{i:rotation} cannot be extended to obtain equality of all rotation numbers: Thompson's group $T$ is $\Z[\tfrac12]/\Z$-combinatorial (Lemma \ref{e:Thompson}), but contains no element of irrational rotation number (Remark \ref{r.Thompson_spectrum}).
\end{rem}

\begin{rem}	
In fact, in this work we will never need to apply Proposition \ref{p:realization lemma} to subgroups $G\subset \Homeo_+(\T)$ containing elements with irrational rotation numbers (see the discussion in Section \ref{ss:multicon_intro}).
Also, condition \eqref{i:close} is needed only when there are non-trivial HNN extensions to consider.
\end{rem}

\begin{proof}[Proof of Proposition \ref{p:realization lemma}]
We argue by induction on the number of edges of the quotient graph $\overline X=X/\alpha(G)$: the marking $(G,\alpha,T)$ gives an identification of $G$ with the fundamental group of a graph of finite cyclic groups (see \eqref{eq:pres_fundamental_group} and Remark \ref{r:mthmC}), and therefore $G$ is obtained by first performing  iterated amalgamated free products of the vertex groups $G_v$, which are finite cyclic groups, and then iterated HNN extensions over the remaining edge groups.

\setcounter{step}{0}
\begin{step}[Initial step]\label{case2_realization}
	$G$ is a finite cyclic group.
\end{step} 

Write $\cG=G=\{\mathrm{id}, s,\dots,s^{q-1}\}$, with $\rot(s)=1/q$, and $P_1=G(P)\subset D_0$ (which is a $G$-invariant finite subset).
Fixing a point $x_1\in P_1$, we can take a monotone enumeration of the points of $P_1$:
\begin{align*}
	x_1&<\dots<x_k< x_{k+1}=s(x_1)<\dots<x_{2k}<x_{2k+1}=s^2(x_1)\\
	&<\dots <x_{(q-1)k+1}=s^{q-1}(x_1)<\dots <x_{qk}\,(<x_1).
\end{align*}
Let $g_0\in G_0$ be an element 
of order $q$ (which exists by condition (C~\ref{order})), which we can assume
to have the same rotation number as $s$. 
Observe that, given $\eps>0$, there exists $g_\eps\in G_0$ with the same properties and which is $\eps$-close to $s$: indeed, let $\varphi\in \Homeo_+(\T)$ be such that $\varphi g_0 \varphi^{-1}=s$, and using Lemma \ref{l.combinatorial_dense}, take $f\in G_0$ which is sufficiently close to $\varphi$, so that the conjugate $g_\eps:=fg_0f^{-1}$ (which is an element of $G_0$) is $\eps$-close to $s$.

Next, for every sufficiently small $\eps_1>0$, there exist a subset $Q_1=\{y_{i}\}_{i=1}^{qk}\subset D_0$ and $\eps_0>0$ such that:
\begin{itemize}
	\item for every $i\in\{1,\ldots,qk\}$, the point $y_i$ is $\frac{\eps_1}{2}$-close to $x_i$ and the map $P_1\to Q_1$ defined by $x_i\mapsto y_i$ is a circular order preserving bijection.
	\item $g_{\eps_0}(y_i)=y_{i+k}$ for every $i\in\{1,\ldots,qk\}$ (where subscripts are taken modulo ${qk}$). Note that here we use condition (C  \ref{invariance}).
\end{itemize}
For every $i\in \{1,\ldots,qk\}$, let $I_i\subset \T$ denote the shortest interval whose endpoints are $x_i$ and $y_i$, and take a finite subset $B\subset D_0\setminus \bigcup_{i=1}^{qk}I_i$ which is $\frac{\eps_1}{2}$-dense in $\T\setminus  \bigcup_{i=1}^{qk}I_i$.	
Then we can use condition (C~\ref{interpolation}) to find an element $h\in G_0$ such that $h(x_i)=y_i$ for every $i\in\{1,\ldots,qk\}$ and $h(x)=x$ for every $x\in B$.
This gives that $h$ is $\eps_1$-close to $\id$ in the uniform distance. Therefore, for every $\eps>0$ there exists $\eps_1>0$ such that the element $\bar g=h^{-1}g_{\eps_0}h$ belongs to $G_0$, has order $q$ and $\rot(\bar g)=1/q$, and is $\eps$-close to $s$. Moreover, by construction, for every $i\in \{1,\ldots,qk\}$ and $j\in\{0,\ldots,q-1\}$, we have  $\bar g^j(x_i)=x_{i+jk}=s^j(x_i)$ (where subscripts are taken modulo ${qk}$). Hence we can define the desired morphism $\pi:G\to G_0$ by $\pi(s)=\bar g$.

The first inductive step will allow to retract the spanning tree $T\subset \overline X$ to a single vertex.

\begin{step}[Iterated amalgamated free products]\label{case3_realization}
	$G=G_1*_CG_2$ is the amalgamated free product of a non-trivial group $G_1$ and a non-trivial finite cyclic group $G_2$,
	over a finite cyclic subgroup $C$
\end{step}

Consider the corresponding system of generators
$\cG_1$  for $G_1$, so that $\cG=\cG_1\cup G_2$. We also write $C=\alpha_e(C)=\omega_e(C)$ for the common subgroup of $G_1$ and $G_2$ in $G_1*_CG_2$. We will show by induction that the result holds for $G$. 
After Step \ref{case2_realization}, and induction on Step \ref{case3_realization}, we can assume that the result holds for both $G_1$ and $G_2$,  for any $\eps>0$, and any finite subset $P\subset D_0$ such that $\cG_1(P)$ and $G_2(P)$ are in $D_0$. (Note here that we do not technically require $\cG_1^2(P)\subset D_0$.) We will show that the same result holds for the amalgamated free product $G_1*_CG_2$.

For this, let us assume that $\cG(P)\subset D_0$.
By induction, there exist morphisms
$\pi_1:G_1\to G_0$ and $\pi_2:G_2\to G_0$ 
satisfying conditions (\ref{i:rotation}-\ref{i:restriction}) in the statement, with respect to the subset $P$ and some sufficiently small $\eps_0>0$ to be fixed later.
Let $q$ be the order of the cyclic subgroup $C$, and let $c$ be a generator of $C$; write $a=\pi_1(c)$ and $b=\pi_2(c)$, and note that they are both $\eps_0$-close to $c$, and thus $2\eps_0$-close one to the other.
By the induction hypothesis, we have that $\rot(a)=\rot(b)$ and the equalities
\[
a^j\restriction_{C(P)}=b^j\restriction_{C(P)}=c^j\restriction_{C(P)}\quad\text{for every }j\in\{0,\ldots,q-1\}.
\]
It follows by condition~(C \ref{conjugation}) that there exists $h \in G_0$ such that
$a=hbh^{-1}$ and $h=\mathrm{id}$ on $C(P)$. Moreover by Lemma \ref{l.C4}, for any fixed $\eps_1>0$ we can take $\eps_0>0$ such that $h$ is $\eps_1$-close to $\id$.

We then consider the conjugate morphism $\bar \pi_2:G_2\to G_0$ given by
$g\mapsto h\pi_2(g)h^{-1}$.  For every $\eps>0$, we can choose $\eps_1>0$ small enough, so that $\bar{\pi}_2(g)$ is $\eps$-close to $g$ for every $g\in G_2$.
We deduce also from \eqref{i:restriction} that for any $s\in G_2$  we have
that $\bar{\pi}_2(s)\restriction_{P}=s\restriction_{P}$. 
Therefore also the morphism $\bar\pi_2$ satisfies conditions (\ref{i:rotation}-\ref{i:restriction}) in the statement, with respect to the subset $P$.

Since
$\bar{\pi}_2(c)=hbh^{-1}=a=\pi_1(c)$, by the universal property of amalgamated  free products, there exists a (unique) morphism $\pi:G\to G_0$ extending $\pi_1$ and $ \bar{\pi}_2$. From the construction, we have $\pi(s)=s$ on $P$ for any $s\in \cG$. This concludes the construction in this step.

Therefore, by induction, we are reduced to consider a graph of groups with a single vertex.		
We remark again that after this first iterated process, we still only require $\cG(P)\subset D_0$ (but not $\cG^2(P)\subset D_0$). The second inductive step concludes the construction.

\begin{step}[Iterated HNN extensions]\label{case4_realization}
	$G=G_1*_C$ is an HNN extension of a group $G_1$ over a finite cyclic group $C$.
\end{step}

Write $C_1=\alpha_e(C)$ and $C_2=\omega_e(C)$, for the two subgroups of $G_1$, and let $s\in G$ be a stable letter, that is, a non-trivial element such that $s\alpha_e(c)s^{-1}=\omega_e(c)$ for every $c\in C$. The system of generators for $G$ is of the form
$\cG=\cG_1\cup \{s,s^{-1}\}$ where $\cG_1$ is a system of generators of
$G_1$ containing $C_1$ and
$C_2$. We assume by induction that the result holds for $G_1$, for any $\eps>0$,  and for any   finite subset $P\subset D_0$ such that $\cG_1^2(P)\subset D_0$. We will show that the result holds for the HNN extension $G_1*_C$.

For this, let us assume that $\cG^2(P)\subset D_0$. We also write $Q=\cG(P)$.
By induction, there exists a morphism
$\pi_1:G_1\to G_0$ satisfying conditions (\ref{i:rotation}-\ref{i:restriction}) in the statement, with respect to the subset $P$ and some sufficiently small $\eps_0>0$ to be fixed later. Moreover, from the previous steps, as $\cG(Q)\subset D_0$, the result also holds for every vertex group $G_v$ (in particular for $C_1$ and $C_2$), with respect to the subset $Q$. Let $r$ be the order of $C$, and take the generators $c_1$ and $c_2$  of $C_1$ and $C_2$ respectively, such that $\rot(c_1)=\rot(c_2)=1/r$. It follows that $sc_1s^{-1}=c_2$.

Assume first that $s$ has rational rotation number, with $\rot(s)=p/q$ in reduced terms. Let $x_0\in \T$ be a periodic point for $s$ and write $x_i=s^i(x_0)$ for $i\in\{1,\ldots,q\}$ (note that $x_q=x_0$). As $s$ has infinite order, there exists a point $y_0\in \T$ which is not periodic, and write similarly $y_i=s^i(y_0)$ for $i\in\{1,\ldots,q\}$ (note that $y_0\neq y_q$).

For any $i\in \{0,\ldots,q\}$, take points $x_i'$ and $y_i'$ in $D_0$ as follows: if $x_i$ (resp.\ $y_i$) belongs to $D_0$, then set $x_i'=x_i$ (resp.\ $y_i'=y_i$); if not, choose $x_i'$ (resp.\ $y'_i$) in $D_0\setminus \cG(Q)$ (with $x_q'=x_0'$) sufficiently close to $x_i$ (resp.\ $y_i$), so that at the end, the map $\cG(Q)\cup \{x_i,y_i\}_{i=0}^{q}\to \cG(Q)\cup\{x'_i,y'_i\}_{i=0}^{q}$	given by
\begin{equation}\label{eq:cond_HNN0}
	\left\{\begin{array}{lr}
		x_i\mapsto x_i' & \text{for }i\in\{0,\ldots, q-1\},\\
		y_i\mapsto y_i' & \text{for }i\in\{0,\ldots, q\},\\
		x\mapsto x& \text{for }x\in \cG(Q),
	\end{array}\right.
\end{equation}
is a circular order preserving bijection.
Note that any element $g\in G_0$ such that
\begin{equation}\label{eq:cond_HNN1}
	\text{$g(x'_i)=x'_{i+1}$ and $g(y'_{i})=y'_{i+1}$ for every $i\in \{0,\ldots,q-1\}$}
\end{equation} (which exists by condition (C \ref{interpolation})), satisfies the requirements in condition \eqref{i:rotation} of the statement:  such an element $g$ has a periodic orbit of order $q$ (the orbit of $x'_0$) and an orbit with more than $q$ points (the orbit of $y_0'$), and therefore it has infinite order; moreover the rotation numbers of $g$ and $s$ are the same.

We need however to choose an element $g\in G_0$ satisfying extra conditions. First of all, we require
$g=s$ on $\cG(Q)\subset D_0$. As the image of $Q \cup s^{-1}(Q)\subset \cG(Q)$ by $g$ is $s(Q)\cup Q\subset  D_0$, we deduce that $g^{-1}=s^{-1}$ on $Q$ (and thus on $P$). This is always possible, and compatible with the previous requirement, by the fact that the bijection defined by \eqref{eq:cond_HNN0} preserves the circular order, so that we can use condition (C \ref{interpolation}).

We also need a strengthening of condition  \eqref{eq:cond_HNN1}: for every $j\in\{0,\ldots,r-1\}$ and $i\in\{0,\ldots, q-1\}$, we want
\begin{equation}\label{eq:cond_HNN2}
	g\pi_1(c_1^j)(x'_i)=\pi_1(c_2^j)(x'_{i+1})\quad\text{and}\quad	g\pi_1(c_1^j)(y'_i)=\pi_1(c_2^j)(y'_{i+1}).
\end{equation}
Let us verify that condition \eqref{eq:cond_HNN2} is realizable. First, we note that as every $x_i'$ and $y_i'$ are in $D_0$, by condition (C \ref{invariance}) also their images by $\pi_1(C_1)$ and $\pi_1(C_2)$ are in $D_0$. 
It remains to justify that there exists $\eps_0>0$ such that the circular ordering of these points is compatible with the requirement \eqref{eq:cond_HNN2}.		
This is because, as  $\eps_0\to 0$, for every $\xi\in \{x_i',y_i'\}_{i=0}^q$ we have
\[
\pi_1(c_1^j)(\xi)\to c_1^j(\xi)\quad \text{and}\quad \pi_1(c_2^j)s(\xi)\to c_2^js(\xi)=sc_1^j(\xi).
\]
Moreover, up to add a finite subset $B\subset D_0$ for defining $g$ by interpolation, as we did in Step \ref{case2_realization}, we can find the desired $g$ which is arbitrarily close to $s$.

Write now $a=g \pi_1(c_1)g^{-1}$ and $b=\pi_1(c_2)$. Then the elements $a$ and
$b$ belong to $G_0$, they have finite order $q$ and same rotation number. As we can assume that $g$ is arbitrarily close to $s$, it follows that $a$ is arbitrarily close to $s\pi_1(c_1)s^{-1}$, which in turn is arbitrarily close to $sc_1s^{-1}=c_2$ and thus to $b=\pi_1(c_2)$.

Since for all $x\in P\cup s(P)$ and $j\in \{1,\ldots,q-1\}$, we have $s^{-1}(x)\in \cG(P)=Q$ and $c_1^j(s^{-1}(x))\in \cG^2(P)=\cG(Q)$, then the following equality holds:
\begin{align*}
	a^j(x)&=g\pi_1(c^j_1)g^{-1}(x)=gc^j_1s^{-1}(x)=sc_1^js^{-1}(x)\\
	&=c^j_2(x)=\pi_1(c^j_2)(x)=b^j(x).
\end{align*}
Considering also \eqref{eq:cond_HNN2}, by condition (C \ref{conjugation}), there exists $h\in G_0$ such
that $hah^{-1}=b$ and $h=\mathrm{id}$ on $P\cup s(P)\cup \{x_i',y_i'\}_{i=0}^{q}$.  As we can assume that $a$ and $b$ are arbitrarily close, Lemma \ref{l.C4} ensures that we can find such an element $h\in G_0$ which is arbitrarily close to $\id$.

Hence, setting
$\bar g=hg$, which belongs to $G_0$ and is arbitrarily close to $s$, we have
\[
\pi_1(c_2)=b=hah^{-1}=hg\pi_1(c_1)g^{-1}h^{-1}=\bar g \pi_1(c_1)
\bar{g}^{-1},
\]
so by the universal property of HNN extensions there exists a (unique) morphism
$\pi:G\to G_0$ extending $\pi_1$ such that $\pi(s)=\bar{g}$.
Moreover, by construction, we have that $\bar{g}=s$ and $\bar{g}^{-1}=s^{-1}$ on
$P$, it follows that $\pi(t)=t$ on $P$ for any $t\in \cG$.
Finally, since $h=\id$ on $\{x_i',y_i'\}_{i=0}^{q}$, 
the element $\bar g$ satisfies the condition in \eqref{eq:cond_HNN1}, and therefore $\rot(\bar g)=\rot(s)$ and $\bar g$ has infinite order.		
Thus the result also holds for the HNN extension $G$, in the case $\rot(s)$ is rational.

Assume now that $s$ has irrational rotation number $\rot(s)=\alpha$, with rational approximations $\{p_n/q_n\}_{n\in\N}$. Take a point $x_0\in \T$ in the minimal invariant set of $s$; observe that (by definition of rational approximations) for any $n\in \N$, there is an interval $I_n\subset \T$, whose endpoints are $x_0$ and $s^{q_n}(x_0)$, and which contains no other point $s^i(x_0)$ with $i\in\{1,\ldots,q_n-1\}$. Choosing $n$ large enough, we can assume that $I_n$ contains no point of $\cG(Q)$ neither. As before, we write $x_i=s^i(x_0)$ for every $i\in\{1,\ldots,q_n\}$.
Take now $y_0\in I_{n-1}\setminus s^{-q_n}(I_n)$, so that  none of the points $y_i:=s^i(y_0)$ for $i\in \{0,\ldots, q_n\}$ is contained in $I_n$ (this follows again by combinatorial properties of rational approximations). 
For any $i\in\{0,\ldots,q_n\}$, take points $x_i'$ and $y_i'$ as follows: if $x_i$ (resp.\ $y_i$) belongs to $D_0$, then set $x_i'=x_i$ (resp.\ $y_i'=y_i$); if not, choose $x_i'$ (resp.\ $y'_i$) in $D_0\setminus \cG(Q)$ sufficiently close to $x_i$ (resp.\ $y_i$), so that at the end, the map $\cG(Q)\cup \{x_i,y_i\}_{i=0}^{q_n}\to \cG(Q)\cup\{x'_i,y'_i\}_{i=0}^{q_n}$ given by
\[
\left\{\begin{array}{lr}
	x_i\mapsto x_i' & \text{for }i\in\{0,\ldots, q_n\},\\
	y_i\mapsto y_i' & \text{for }i\in\{0,\ldots, q_n\},\\
	x\mapsto x& \text{for }x\in \cG(Q),
\end{array}\right.
\]
is a circular order preserving bijection.
Thus, by condition (C \ref{interpolation}), there exists an element $g\in G_0$ such that the following properties are satisfied.
\begin{itemize}
	\item $g=s$ on $\cG(Q) \subset D_0$. As before, this gives $g^{-1}=s^{-1}$ on $Q$.
	\item $g(x'_i)=x'_{i+1}$ for every $i\in \{0,\ldots,q_{n}-2\}$, and $g(x'_{q_n-1})=x'_0$. (Note that this implies $\rot(g)=p_n/q_n$.)
	\item $g(y'_{i})=y'_{i+1}$ for every $i\in \{0,\ldots,q_n-1\}$.	Therefore $g$ has infinite order.
\end{itemize}
We then proceed as in the case of rational rotation number: we take $g\in G_0$ which satisfies in addition the condition analogue to \eqref{eq:cond_HNN2}, from which we will be able to find $h\in G_0$ such
that $hah^{-1}=b$ and $h=\mathrm{id}$ on $P\cup s(P)\cup \{x_i',y_i'\}_{i=0}^{q_n-1}$. We then define $\pi(s)=hg$ as before, which will have the desired properties. Details are left to the reader.

Therefore, after repeating Step \ref{case4_realization} finitely many times, we get that for every $\eps>0$, and every finite subset $P\subset D_0$ such that $\cG^2(P)\subset D_0$, there exists $\pi:G\to G_0$ satisfying conditions (\ref{i:rotation}-\ref{i:restriction}) of the statement.	
\end{proof}

\begin{thm}\label{t:realization theorem}
Let $(G,\alpha,T)\subset \Homeo_+(\T)$ be a marked
virtually free group whose action on $\T$ admits a proper ping-pong partition. Let $D_0\subset \T$ be a dense subset and let $G_0$ be a $D_0$-combinatorial group of circle homeomorphisms. Then the group $G$ is semi-conjugate to a subgroup of $G_0$. Moreover, the semi-conjugacy can be chosen arbitrarily close to the identity with respect to the uniform distance.
\end{thm}

\begin{proof}
Let $\cG$ be the system of generators associated with the marking $\alpha:G\to\Isom_+(T)$. Denote by $\Theta$ a proper ping-pong partition and let $\Delta$ be the collection of endpoints of the atoms of the partition $\Theta$. By density of $D_0$, there exists  an orientation preserving circle homeomorphism $\theta:\T\to\T$, such that $\theta(\mathcal{G}^2(\Delta))\subset D_0$,  and with $\theta$ arbitrarily close to $\id$ in the uniform distance. Thus, upon conjugating by $\theta$, we can assume that $\mathcal{G}^2(\Delta)\subset D_0$. It follows  from Proposition~\ref{p:realization lemma} that for every $\eps>0$ there exists a morphism $\pi:G\to G_0$ satisfying conditions (\ref{i:rotation}-\ref{i:restriction}) of its statement. In particular, for every $s\in \cG$, the image $\pi(s)$ is $\eps$-close to $s$  and  we have 
\begin{equation}\label{eq:restriction_ping_pong}
	\pi(s)\restriction_{\Delta}=s\restriction_\Delta.
\end{equation} 
It easily follows that $\Theta$ is also a proper interactive family for $\pi(G)$, and that $\pi:G\to G_0$ is injective (see \cite[Theorem 5.4]{MarkovPartitions1}). 
Using \eqref{eq:restriction_ping_pong}, we deduce that $(G,\alpha,T)$ and $(\pi(G),\alpha\circ \pi^{-1},T)$ admit equivalent proper ping-pong partitions.
Thus,
the subgroups $G$ and $\pi(G)$ are semi-conjugate by Theorem~\ref{t:bpvfreeB}. Moreover, as we have that $s$ and $\pi(s)$ are $\eps$-close for every $s\in \cG$, it follows that we can choose the semi-conjugacy arbitrarily close to the identity in the uniform distance.
\end{proof}

\begin{proof}[Proof of Theorem~\ref{t:Dippolito}]
This follows directly from Lemma \ref{e:Thompson}, which states that Thompson's group $T$ is $\Z[\frac12]$-combinatorial, Theorem \ref{t:bpvfreeA}, which gives a proper ping-pong partition, and Theorem~\ref{t:realization theorem}, which gives the realization in $T$.
\end{proof}

\subsection{Real-analytic regularity}

In this section we will show that the group of real-analytic diffeomorphisms $\Diff_+^\omega(\T)$ is combinatorial. We start with an interpolation result in real-analytic regularity, from which we will deduce condition (C \ref{interpolation}).  This will be also used for the proof of Theorem \ref{t.minimal_actions}.

\begin{lem}\label{l:interpolationCw}
Given finitely many points $0\le x_1<\dots<x_n<1$ and $0\le y_1<\dots<y_n<y_1+1$ there exists a real-analytic function $f\in C^\omega(\R)$ such that $f'(x)>0$ for all $x\in\R$, $f(x)-x$ is $1$-periodic and for all $j\in\{1,\ldots,n\}$ the following conditions hold:
\begin{enumerate}[\rm (1)]
	\item\label{interpolate} $f(x_j)=y_j$;
	\item\label{derivative} $f'(x_j)=1$;
	\item\label{derivative2} if $y_{j+1}-y_j>x_{j+1}-x_j$, then  $f'(x)>1$ for every $x\in (x_j,x_{j+1})+\Z$;
	\item\label{derivative3} if $y_{j+1}-y_j<x_{j+1}-x_j$, then  $f'(x)<1$ for every $x\in (x_j,x_{j+1})+\Z$.
\end{enumerate}
\end{lem}

\begin{proof} Using a trigonometrical approximation, we shall find a real-analytic function with the required properties. We will first construct the derivative in the form $f'(x)=1+q(x)r(x)$ where $q$ is a $1$-periodic trigonometric polynomial with roots at $x_1,\ldots,x_n$ and $r$ is everywhere positive. More precisely, consider the subset $S\subset \{1,\ldots,n\}$ formed by all $j\in \{1,\ldots,n\}$ such that
\[
\frac{y_j-y_{j-1}}{x_j-x_{j-1}}-1\quad\text{and}\quad\frac{y_{j+1}-y_j}{x_{j+1}-x_j}-1
\]
have distinct sign. Then set $\epsilon_j=1$ if $j\in S$, and $\epsilon_j=2$ if $j\in \{1,\ldots,n\}\setminus S$. For an appropriate choice of sign, the trigonometric polynomial
$q(x)=\pm\prod_{j=1}^n\sin^{\epsilon_j}\left (\pi(x-x_j)\right )$ has the desired properties.

Let us look for a $1$-periodic function $r\in C^\omega(\R)$ of the form $r(x)=\sum_{j=1}^n S_j\, r_j(x)$, where $S_j>0$ and the $r_j\in C^\omega(\R)$ will be constructed as approximate characteristic functions of the subsets $(x_j,x_{j+1})+\Z$. 
Given a sufficiently small $\eps>0$ and $j\in\{1,\ldots,n\}$, consider the following $1$-periodic function $\varrho_j^\eps\in C^0(\R)$:	
\begin{equation*}\label{eq:r_j_epsilon}
	\varrho_j^\eps(x)=
	\left \{\begin{array}{lr}
		\eps &\text{for } x\not\in [x_j, x_{j+1}]+\Z,\\[0.5em]
		\eps+\left(\dfrac1{|q(x_j+\eps)|}-\eps\right)\dfrac{x-x_j-m}\eps &\text{for } x\in [x_j, x_j+\eps]+m,\\[0.5em]
		\dfrac1{|q(x)|} &\text{for } x\in [x_j+\eps, x_{j+1}-\eps]+m,\\[0.5em]
		\eps-\left(\dfrac1{|q(x_{j+1}-\eps)|}-\eps\right)\dfrac{x-x_{j+1}-m}\eps &\text{for } x\in [x_{j+1}-\eps, x_{j+1}]+m,
	\end{array}\right.  (m\in \Z).
\end{equation*}
Using Weierstrass approximation theorem, we can take a $1$-periodic trigonometric polynomial  $r_j^\eps\in C^\omega(\R)$ which is $\eps/2$-close to $\varrho_j^\eps$ (with respect to the uniform distance). Note that after this choice, the function $r_j^\eps$ is everywhere positive for sufficiently small $\eps>0$.
Also, by construction, the function $|q|\cdot r_j^\eps$ tends pointwise to $\mathbf{1}_{(x_j,x_{j+1})+\Z}$ as $\eps\to 0$, and therefore the quantity $a_{jk}^\eps:=\int_{x_k}^{x_{k+1}} |q(x)|r^\eps_j(x)\, dx$ (where we set $x_{n+1}=x_1+1$) tends to $(x_{k+1}-x_k)\delta_{jk}$ as $\eps\to 0$.

In order to ensure property \eqref{interpolate} we consider the following system of linear equations in the unknowns $S_1,\ldots,S_n$:
\begin{equation}\label{eq:S_j}
	\sum_{j=1}^n S_j\, a_{jk}^\eps=\left\vert(y_{k+1}-y_{k})-(x_{k+1}-x_k)\right\vert\quad (k\in \{1,\ldots,n\}),
\end{equation}
where we write $y_{n+1}=y_1+1$. By the previous observation, as $\eps\to 0$, the matrix of the system \eqref{eq:S_j} tends to the diagonal matrix $\mathsf{diag}(x_2-x_1,\ldots, x_{n+1}-x_n )$, and therefore for sufficiently small $\eps>0$, the system  \eqref{eq:S_j} admits a unique solution, which we denote by $S_1^\eps,\ldots,S_n^\eps$. Note that for every $j\in\{1,\ldots,n\}$, we have
\[
S_j^\eps\to \left\vert \dfrac{y_{j+1}-y_j}{x_{j+1}-x_j}-1\right\vert\quad\text{as }\eps\to 0.
\]
Set then $r^\eps(x):=\sum_{j=1}^nS^\eps_j\,r^\eps_j(x)$, and observe that  for every $x\in (x_j,x_{j+1})+\Z$ we have
\[
1+q(x)r^\eps(x)\to\frac{y_{j+1}-y_j}{x_{j+1}-x_j}\quad\text{as }\eps\to 0.
\]
Therefore, we can find $\eps>0$ such that the following conditions are satisfied:
\begin{itemize}
	\item the system (\ref{eq:S_j}) has a solution (hence the function $r^{\eps}$ is defined),
	\item for every $x\in \R$, we have $1+q(x)r^{\eps}(x)>0$, and moreover $1+q(x)r^{\eps}(x)>1$ (respectively, $<1$) for every $x\in (x_j,x_{j+1})+\Z$, where $j\in \{1,\ldots,n\}$ is such that $y_{j+1}-y_j>x_{j+1}-x_j$ (respectively, $y_{j+1}-y_j<x_{j+1}-x_j$).
\end{itemize} 
We can then take $r=r^{\eps}$ and define
\[f:x\in\R\mapsto y_1+\int_{x_1}^x 1+q(\theta)r(\theta)\, d\theta,\]
which gives the desired function.
\end{proof}

The next statement is the main result of this section.

\begin{prop}\label{r:analytic is combinatorial}
The group of real-analytic diffeomorphisms $\Diff_+^\omega(\T)$ is combinatorial.
\end{prop}

\begin{proof}
As we are considering $D_0=\T$, condition (C \ref{invariance}) is trivially satisfied. Lemma \ref{l:interpolationCw} above gives condition (C \ref{interpolation}). Condition (C \ref{order}) is clear (all rigid rotations are in $\Diff_+^\omega(\T)$).
It is well-known  that any two  real-analytic diffeomorphisms of finite order and same rotation number are $C^\omega$ conjugate (see e.g.\ \cite[II.6]{Herman}): let $a\in \Diff^\omega_+(\T)$ be of finite order and $\rot(a)=p/q$, denote by $A$ any lift of $a$ to the real line, then $h_a(x)=\frac{1}{q}\sum_{i=0}^{q-1}A^i(x)-i\frac{p}{q}\pmod\Z$ defines a $C^\omega$ conjugacy of $a$ to the rotation $R_{p/q}$ of angle $p/q \pmod\Z$.
Therefore given any two finite order elements $a$ and $b\in \Diff_+^\omega(\R)$ with same rotation number $\rot(a)=\rot(b)=p/q$, the composition $h=h_b^{-1}\circ h_a$ conjugates $a$ to $b$. We claim that it also satisfies the additional requirements in condition (C \ref{conjugation}). For this, let  $E\subset \T$ be a finite subset such that $a^i(x)=b^i(x)$ for every $i\in\{1,\ldots,q-1\}$, and denote by $A$ and $B$ two lifts of $a$ and $b$ respectively, to the real line. Then for every $x\in E$, we have
\[
h_a(x)=\frac{1}{q}\sum_{i=0}^{q-1}A^i(x)-i\frac{p}{q}\pmod\Z=\frac{1}{q}\sum_{i=0}^{q-1}B^i(x)-i\frac{p}{q}\pmod\Z=h_b(x)
\]
and thus $h(x)=h_b^{-1}h_a(x)=x$.
This gives condition (C \ref{conjugation}).
\end{proof}

\subsection{Minimality criterion}
Ping-pong partitions are an adapted tool to recognize semi-conjugate actions, but they do not provide full topological information. In particular,  it is unclear how to decide whether an action with ping-pong partition is minimal. In this direction we prove the following partial result.

\begin{thm}\label{t.minimal_actions} Every marked, finitely generated free subgroup of  $\Homeo_+(\Sc)$ with a proper ping-pong partition  is semi-conjugate to a subgroup of $\Diff^\omega_+(\Sc)$ whose action is minimal.
\end{thm}

We will prove Theorem \ref{t.minimal_actions} later in the subsection. Before this, we will make some general considerations for recognizing minimal actions of virtually free groups.

\begin{rem}
At the present moment we are unable to extend Theorem \ref{t.minimal_actions} to every \emph{virtually} free group; technical difficulties arise when there are torsion elements. In fact, it could be that not every combinatorial configuration is realizable in real-analytic regularity: an elementary instance of this is Kontsevich's ``m\'etro ticket'' theorem, discussed  by Ghys in \cite{promenade}, and our approach is faced to an obstruction of similar nature. In some particular cases, it is possible to construct a minimal action, for instance when the group is a central extension of a free group, or when the group is of the form $F_n\rtimes \Z_m$, where the cyclic group acts on $F_n$ by permutation of a free basis (in which case one has simply to implement the symmetries of the free generating set in their real-analytic realizations as in Lemma \ref{l:interpolationCw}).
\end{rem}

\subsubsection{Ping-pong partitions and Markovian sequences}
Let $(G,\alpha,T)$ be a marked, finitely generated, virtually free subgroup of $\Homeo_+(\Sc)$ acting with a ping-pong partition $\Theta$. Recall that $\cI$ denotes the collection of connected components of atoms of $\Theta$ and $\cJ$ the collection of gaps of the partition. It is proved in \cite[Proposition 6.5]{MarkovPartitions1} that the minimal set $\Lambda$ for $G$ is contained in the union of the closures of intervals of $\cI$. Hence for such an action to be minimal, it is necessary that the gaps of the partition (they do not need to be real gaps of $\Lambda$) be degenerate (i.e.\ reduced to points). Unfortunately this condition is not sufficient as the next example shows.

\begin{ex}\label{ex.fake_ppp}
Consider a classical Schottky configuration.
For that we take a circularly ordered family of open arcs $\cI_0=\{I_s,I_t,I_{s^{-1}},I_{t^{-1}}\}$ in $\Sc$, with pairwise disjoint closure, and we denote by $\cJ_0$ the set of gaps of $\cI_0$ indexed as follows. For $\gamma\in\{s,s^{-1},t,t^{-1}\}$, denote by $J_\gamma$  the element of $\cJ_0$ located on the left of $I_\gamma$. Consider two homeomorphisms $s$ and $t$ in $\Homeo_+(\Sc)$ such that 
\begin{align*}
	s(I_s)&=I_t\cup J_s\cup I_s\cup J_{t^{-1}}\cup I_{t^{-1}},\\
	t(I_t)&=I_{s^{-1}}\cup J_t\cup I_t\cup J_{s}\cup I_{s},\\
	s^{-1}(I_{s^{-1}})&=I_{t^{-1}}\cup J_{s^{-1}}\cup I_{s^{-1}}\cup J_{t}\cup I_{t},\\
	t^{-1}(I_{t^{-1}})&=I_{s}\cup J_{t}\cup I_{t^{-1}}\cup J_{s^{-1}}\cup I_{{s^{-1}}}.
\end{align*}
See Figure \ref{fig:schottky}. Then the group $G$ generated by $S=\{s,s^{-1},t,t^{-1}\}$ is free. The action of $G$ has an exceptional minimal set $\Lambda$ and $\cI_0$ defines a ping-pong partition. For $\gamma\in S$, define $U_\gamma=\mathsf{Int}(J_\gamma\cup I_\gamma)$.  It is a simple exercise to verify that, although $G$ has an exceptional minimal set, $\Theta=\{U_s,U_t,U_{s^{-1}},U_{t^{-1}}\}$ is a ping-pong partition for the action of $G$ with degenerate gaps. 

\begin{figure}
	\[
	\includegraphics{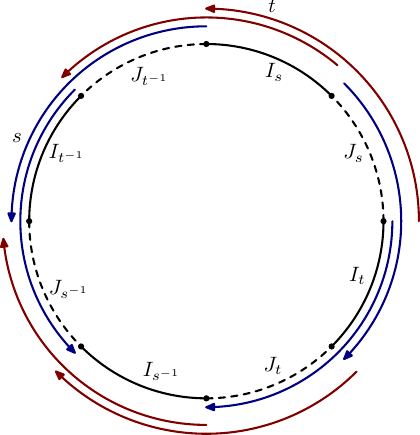}\quad\quad \includegraphics{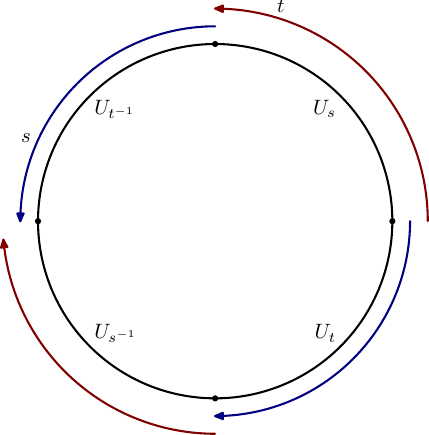}\]
	\caption{Partitions for Example \ref{ex.fake_ppp}.}
	\label{fig:schottky}
\end{figure}
\end{ex}

This is an example of proper ping-pong partition with \emph{neutral intervals} (the intervals $J_\gamma$) which motivates the following definition.

\begin{dfn}[Markovian sequences and neutral intervals] Let $(G,\alpha,T)$ be a marked, finitely generated, virtually free subgroup of $\Homeo_+(\Sc)$ with ping-pong partition $\Theta$. We let $\cG$ be the corresponding generating set of $G$.
A \emph{Markovian sequence} for an interval $I\in\cI$ is a finite sequence of generators $s_1,\ldots, s_k\in \cG$ such that for every $l\in\{1,\ldots,k-1\}$, the image $s_l s_{l-1}\cdots s_1(I)$ is an interval of $\cI$ and the image $s_k(s_{k-1}\cdots s_1(I))$ is $\cI$-Markovian.
An interval $I\in\cI$ without Markovian sequence is called \emph{neutral}.
\end{dfn}

\noindent
\begin{minipage}[t]{.45\textwidth}
\centering\raisebox{\dimexpr \topskip-\height}{%
	\includegraphics{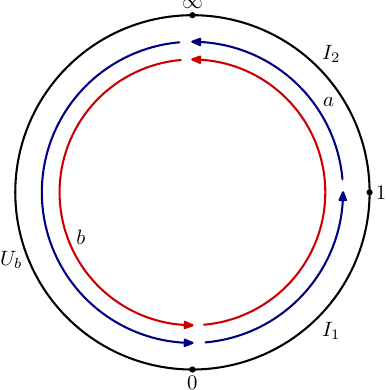}}
\captionof{figure}{Ping-pong partition for $\PSL(2,\Z)\cong\Z_2*\Z_3$ (Example \ref{ex:Farey}).}
\label{fig:Farey}
\end{minipage}\hfill
\begin{minipage}[t]{.55\textwidth}
\begin{ex}\label{ex:Farey}
	Consider the classical Farey partition $\Theta=\{U_a,U_b\}$ for the action of $\PSL(2,\Z)\cong \Z_2*\Z_3$, where $a$ and $b$ denote generators of order 3 and 2 respectively. Considering the parametrization of the circle $\T\cong \R\cup \{\infty\}$, the subset $U_a$ is the union of the two disjoint intervals $I_1=(0,1)$ and $I_2=(1,\infty)$, whereas $U_b=(\infty,0)$ is a single interval. The gaps of the partition are the points $0$, $1$, and $\infty$.
	We have the relations
	\begin{align*}
		b(U_b)&= I_1\cup \{1\}\cup I_2,\\
		a(I_1)&= I_2,\\
		a(I_2)&= U_b.
	\end{align*}
	See Figure~\ref{fig:Farey}. A Markovian sequence for the interval $U_b$ is given by $b$, while for the intervals $I_1$ and $I_2$, we can take $a^2,b$ and $a,b$ respectively (no sequences of length 1 exist).  
\end{ex}
\end{minipage}

\noindent
\begin{minipage}[t]{.55\textwidth}
\begin{ex}\label{ex:MMRT}
	An interesting example of ping-pong partition with intervals with non-trivial Markovian sequences (that is, admitting no sequences of length 1), appears in \cite{MMRT}. It is a ping-pong partition for an action of the free group $F_2=\langle f,g\rangle$. Take eight circularly ordered points $x_1,\ldots,x_8\in \T$ and for $i\in \{1,\ldots, 8\}$ write $I_i=(x_i,x_{i+1})$ (indices are taken modulo 8). The partition is defined by the following relations:
	\begin{align*}
		f(I_1)&= I_8\cup \{x_1\}\cup I_1\cup \{x_2\}\cup I_2,\\
		g^{-1}(I_2)&=I_7,\\
		f^{-1}(I_3)&= I_2\cup \{x_3\}\cup I_3\cup \{x_4\}\cup I_4,\\
		g^{-1}(I_4)&=I_1\cup\{x_2\}\cup I_2\cup \{x_3\}\cup I_3\cup \{x_4\}\cup I_4\cup \{x_5\}\cup I_5,\\
		f(I_5)&= I_4\cup \{x_5\}\cup I_5\cup \{x_6\}\cup I_6,\\
		g(I_6)&=I_5\cup\{x_6\}\cup I_6\cup \{x_7\}\cup I_7\cup \{x_8\}\cup I_8\cup \{x_1\}\cup I_1,\\
		f^{-1}(I_7)&= I_6\cup \{x_7\}\cup I_7\cup \{x_8\}\cup I_8,\\
		g(I_8)&=I_3.
	\end{align*}
	Markovian sequences for the intervals $I_2$ and $I_7$ are given  by $g^{-1}, f^{-1}$ and $g,f^{-1}$, respectively.
\end{ex}
\end{minipage}\hfill
\begin{minipage}[t]{.45\textwidth}
\centering\raisebox{\dimexpr \topskip-\height}{%
	\includegraphics{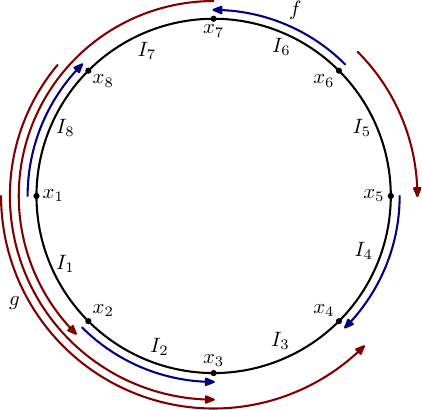}}
\captionof{figure}{A ping-pong partition for an action of $F_2$ with non-trivial Markovian sequence (Example \ref{ex:MMRT}).}
\label{fig:MMRT}
\end{minipage}

\begin{lem}\label{l_gap_neutral}
Let $(G,\alpha,T)$ be a marked, finitely generated, virtually free subgroup of $\Homeo_+(\Sc)$ with proper ping-pong partition $\Theta$. Assume that $I\in\cI$ is neutral. Then $I\cap\Lambda=\varnothing$.
\end{lem}

\begin{proof}
Let $I$ be a neutral interval and let $\cI'\subset \cI$ be the subcollection of intervals $I'\in\cI$ such that there exists a finite sequence $s_1,\ldots,s_k\in\cG$ such that for every $l\in\{1,\ldots,k\}$ we have $I_l=s_l s_{l-1}\cdots s_1(I)\in\cI$ and $I_k=I'$. It is clear that $\cI'$ is a finite family of neutral intervals of $\cI$. As before, we denote by $\Delta\subset \T$ the collection of endpoints of intervals of $\mathcal I$.

For $k\ge 1$, let $\Theta_k$ be the $k$-th refinement of $\Theta$ (that is, we iterate $k$ times the refinement described in \cite[\S 6.3]{MarkovPartitions1}). We can prove by induction that intervals of $\cI'$ are neutral connected components of atoms of $\Theta_k$. Set $\Delta_0=\Delta$ and inductively $\Delta_k=\cG(\Delta_{k-1})$; then  $\Delta_k$ is the collection of endpoints of intervals of $\Theta_k$ (see \cite[\S 6.5]{MarkovPartitions1}). We have $I\cap\Delta_k=\varnothing$ for all $k\ge 0$. Since $I$ is open and $\Delta_\infty=\bigcup_{k\in\N}\Delta_k$ accumulates on $\Lambda$ (see \cite[Proof of Theorem B]{MarkovPartitions1}), we deduce that $\Lambda\cap I=\varnothing$.
\end{proof}

\subsubsection{Strictly Markovian partitions} In order to prove Theorem \ref{t.minimal_actions}, we need the following definition.

\begin{dfn}
Let $(G,\alpha,T)$ be a marked, finitely generated, virtually free subgroup of $\Homeo_+(\Sc)$. We say that a ping-pong partition $\Theta$ for $(G,\alpha,T)$ is \emph{strictly Markovian} if it is proper, its gaps are degenerated, and it has no neutral interval.
\end{dfn}

The next two lemmas show how to reduce to the case of group actions with strictly Markovian ping-pong partitions.

\begin{lem}\label{lem_free_subgrp}
Let $(G,\alpha,T)$ be a marked, finitely generated, virtually free subgroup of $\Homeo_+(\Sc)$ with proper ping-pong partition $\Theta$ and minimal invariant set $\Lambda$. Then for every atom $O\in\Theta$, we have $O\cap\Lambda\neq\varnothing$.
\end{lem}

\begin{proof}
%
As $\Lambda\subset \bigcup_{O\in \Theta}\overline O$, there exists an atom $O\in\Theta$ intersecting $\Lambda$. Now, by the $G$-invariance of $\Lambda$, we can propagate this to every other atom of $\Theta$ using the properties (IF \ref{pp2},\ref{pp7},\ref{pp8}) of an interactive family.
\end{proof}

\begin{lem}\label{lem_collapse}
Every marked, finitely generated, virtually free subgroup of $\Homeo_+(\Sc)$ with a proper ping-pong partition is semi-conjugate to a marked subgroup of $\Homeo_+(\Sc)$ with a strictly Markovian ping-pong partition.
\end{lem}

\begin{proof}
If the action of the subgroup is minimal then the ping-pong partition has no neutral intervals (see Lemma \ref{l_gap_neutral}) and its gaps are degenerated. We assume next that the minimal invariant set $\Lambda$ is a Cantor set. We write $(G,\alpha,T)$ for the marked subgroup, and $\Theta$ for the ping-pong partition.
Take a minimalization of $G$, that is, a subgroup $G'\subset\Homeo_+(\Sc)$, isomorphic to $G$, whose action is minimal and semi-conjugate to that of $G$. Write $h:\T\to\T$ for the surjective map realizing the semi-conjugacy between $G$ and $G'$. Note that $h$ is finite-to-1 on $\Lambda$.
By Lemma \ref{lem_free_subgrp}, for every atom $O\in\Theta$, the intersection $O\cap\Lambda$ is infinite (because $\Lambda$ has no isolated points) so the image $h(O)$ has non-empty interior. 
We define $\Theta'=\{\mathsf{Int}(h(O))\}_{O\in\Theta}$. As $h$ defines a semi-conjugacy, $\Theta'$ defines a ping-pong partition for the action of  $G'$ (with respect to the marking induced by the isomorphism with $G$). Since $G'$ is minimal, we are now reduced to the first case considered in this proof, so the ping-pong partition $\Theta'$ has no non-degenerated gaps and no neutral intervals.
\end{proof}

\subsubsection{Markovian expansion and first returns} We assume that $(G,\alpha,T)$ is a marked, finitely generated, virtually free subgroup of $\Homeo_+(\Sc)$ with a strictly Markovian ping-pong partition $\Theta$.
As the ping-pong partition $\Theta$ is strictly Markovian, for any point $x\in\Sc$ there exists a unique interval $I_x\in\cI$ such that $x$ belongs to $I_x$ union its leftmost point.
Fix $x_0\in \T$. We define by induction a sequence of generators $\{s_n\}_{n\in\N}\subset \cG$  and intervals $\{I_n\}_{n\in\N}\subset \cI$, as follows.
Write $I_0=I_{x_0}$ and, for any $n\ge 1$, write $x_n=s_{n-1}(x_{n-1})$ and $I_{n}=I_{x_n}$; for any $n\ge 0$, we take as $s_n\in \cG$ the first generator appearing in a shortest Markovian sequence for $I_n$.
The sequence of partial compositions $\{s_n\cdots  s_0\}_{n\in\N}$ is called a \emph{Markovian expansion} of $x_0\in\Sc$.

Take now $I\in\cI$ and let $x_0\in\Delta_0$ be its leftmost point. Consider  a Markovian expansion  $\{s_n\cdots  s_0\}_{n\in\N}$ of $x_0$ and assume there exists a least integer $n\geq 1$ such that $x_n=x_0$; then the element $h=(s_{n-1}\cdots  s_0)^{-1}$ satisfies $I\supset h(I)$ and $h(x_0)=x_0$. In analogy with Definition \ref{d.wandering}, we say that the element $h$ is a \emph{first return to $I$}.

%

\begin{lem}\label{l_extrem_lambda}
Let $(G,\alpha,T)$ be a marked, finitely generated, virtually free subgroup of $\Homeo_+(\Sc)$ with a strictly Markovian ping-pong partition $\Theta$. Assume that for every $I\in\cI$ and every first return $h$ to $I$, the map $h$ has no fixed point on $I$. Then $\Delta_0\subset \Lambda$.
\end{lem}

\begin{proof}
Let $x_0\in\Delta_0$. We take a Markovian expansion $\{s_n\cdots  s_0\}_{n\in\N}$ of $x_0$. Since $\Delta_0$ is finite, there exist two least integers $n<m$ such that $x_n=x_m$. We write $x=x_n=x_m$ for this point and $I=I_n=(x,y)$ for the corresponding interval in $\cI$. Consider the element $h=(s_{m-1}\cdots s_{n})^{-1}\in G$.
By minimality of the choice of $n$ and $m$, the map $h$ is a first return to $I$. In particular $h$ has no fixed point on $h(I)\cap I$ by hypothesis.
As $h(I)\subset I$, we deduce that $h^n(y)\to x$ as $n\to \infty$. Note that for every $n\in \N$, we have $h^n(y)\in \Delta_\infty=G(\Delta_0)$. As $\Delta_\infty$ accumulates on $\Lambda$ (see \cite[Proof of Theorem B]{MarkovPartitions1}), we deduce that $x\in \Lambda$, and thus, by $G$-invariance, also $x_0=(s_{n-1}\cdots s_0)^{-1}(x)$ is in $\Lambda$.
\end{proof}

\subsubsection{Minimal and real-analytic realization}

We keep the assumption that $(G,\alpha,T)$ admits a strictly Markovian ping-pong partition $\Theta$.
For $l\geq 1$, we let $\cI_l$ be the collection of intervals of $\cI$ admitting a shortest Markovian sequence of length~$l$.
\begin{rem}
By definition, we have that $I\in\cI_1$ if and only if there exists $s\in\cG$ such that $s(I)$ is $\cI$-Markovian. When $l\geq 2$, if $I\in\cI_l$ and $s_1,\ldots,s_l$ is a shortest Markovian sequence, then $s_1(I)\in\cI_{l-1}$.
Since $\cI$ is finite there exists a largest $k$ such that $\cI_k$ is non-empty, and as we are assuming that $\Theta$ is strictly Markovian, we have $\cI=\bigcup_{l=1}^k\cI_l$.
\end{rem}

For practical purposes, it will be better to work with a particular class of strictly Markovian ping-pong partitions:

\begin{dfn}\label{d.geometric}
Let $\Theta$ be a ping-pong partition for a marked, finitely generated, virtually free subgroup $(G,\alpha,T)$ of $\Homeo_+(\Sc)$. 
We say that $\Theta$ is \emph{geometric} if it is strictly Markovian and for every interval $I\in\cI$ and shortest Markovian sequence $s_1,\ldots,s_l$ for $I$, we have $|s_1(I)|>|I|$.
\end{dfn}

Note the definition above is not restrictive:

\begin{lem}\label{lem_geometric}
Every marked, finitely generated, virtually free subgroup $(G,\alpha,T)$ of $\Homeo_+(\Sc)$ with a strictly Markovian ping-pong partition is conjugate to a subgroup $G'$ of $\Homeo_+(\T)$, for which the ping-pong partition induced by the conjugacy is geometric.

Moreover, when $G\subset \Diff_+^\omega(\T)$, we can take $G'\subset \Diff_+^\omega(\T)$.
\end{lem}

\begin{proof}
Let $\Theta$ be the ping-pong partition for $(G,\alpha,T)$, with collection of intervals $\cI$.
We first make the choice of a collection of intervals $\cI'$, and then consider an appropriate conjugating map. For this, we take
a partition $\cI'$ of the circle into $\#\cI$ open intervals, and fix a bijection $\theta:\cI\to\cI'$ which preserves the circular order of the intervals. For $l\in\{1,\ldots, k\}$, we  set $\cI_l'=\theta(\cI_l)$. We choose the partition $\cI'$ so that the following conditions are satisfied:
\begin{gather}
	\label{ii:geometric} \max_{I'\in\cI'}|I'|<2\min_{I'\in\cI'}|I'|;\\
	\label{i:geometric}  \min_{I'\in \cI_{l-1}} |I'|>\max_{I'\in\cI'_l}|I'|\quad\text{for any }l\ge 2.
\end{gather}
We can now take a real-analytic diffeomorphism $h:\Sc\to\Sc$ such that for every $I\in\cI$, we have $h(I)=\theta(I)$. Then the conjugate subgroup $G'=hGh^{-1}$ satisfies the desired properties.
Indeed, if $I'\in \cI'_1$, then there exists a generator $s\in \cG'=h\cG h^{-1}$ such that $s(I')$ is $\cI'$-Markovian, and by condition \eqref{ii:geometric} we have 
\[|s(I')|\ge 2\min_{J'\in \cI'}|J'|>\max_{J'\in \cI'}|J'|\ge |I'|.\]
When $I'\in \cI'_l$, with $l\ge 2$, there exists a generator $s\in \cG'$ such that $s(I')\in \cI'_{l-1}$ and by condition \eqref{i:geometric} we have $|s(I')|>|I'|$.
\end{proof}

\begin{dfn}\label{d.inf_geometric}
Let $\Theta$ be a geometric ping-pong partition for a marked, finitely generated, virtually free subgroup $(G,\alpha,T)$ of $\Diff_+^1(\Sc)$. 
We say that $\Theta$ is \emph{infinitesimally geometric} if for every interval $I\in\cI$ and shortest Markovian sequence $s_1,\ldots,s_l$ for $I$, we have $s_1'(x)>1$ for every $x \in I$.
\end{dfn}

\begin{lem}\label{l.minimality}
Every marked, finitely generated, virtually free subgroup of $\Diff^1_+(\Sc)$ with an infinitesimally geometric ping-pong partition acts minimally.
\end{lem}

\begin{proof}
Let $(G,\alpha,T)$ be such a group and $\Theta$ the infinitesimally geometric ping-pong partition.
We assume by contradiction that the minimal invariant set $\Lambda$ is a Cantor set. Note first that as $\Theta$ is infinitesimally geometric, for every first return $h$ to $I$, we have $h'(x)<1$ for every $x\in I$. This proves, by the mean value theorem, that the first return $h$ has no fixed point on $I$. By Lemma \ref{l_extrem_lambda}, this implies that $\Delta_0\subset\Lambda$. We deduce that all gaps of $\Lambda$ are contained inside the intervals of $\cI$. Let $J_0$ be a gap of $\Lambda$, and let $I_0\in \cI$ be such that $J_0\subset I_0$. By induction, we can find sequences of generators $\{s_n\}_{n\in\N}$, gaps $\{J_n\}_{n\in\N}$ of $\Lambda$, and intervals $\{I_n\}_{n\in\N}$ of $\cI$ such that for every $n\in\N$:
\begin{enumerate}
	\item $J_n\subset I_n$;
	\item either $s_{n}(I_n)=I_{n+1}$, or $s_{n}(I_n)$ is $\cI$-Markovian and $I_{n+1}\subset s_{n}(I_n)$;
	\item $s_{n}'(x)>1$ for every $x\in I_n$;
	\item $s_{n}(J_n)=J_{n+1}$.
\end{enumerate}
From these conditions it is clear that the sequence of lengths $|J_n|$ is strictly increasing, and that the gaps $J_n$ are pairwise disjoint (if two gaps of $\Lambda$ have non-empty intersection they must be equal). This contradicts the fact that there are only finitely many gaps whose length exceeds a given constant.
\end{proof}

From this discussion, we deduce that ping-pong partitions for free groups admit minimal real-analytic realizations.

\begin{proof}[Proof of Theorem \ref{t.minimal_actions}]
In the case of a free group $G$, a marking corresponds to the choice of a symmetric free generating system $S$ (see \S\ref{ss.Free_DKN}). We also choose a subset $S_0\subset S$ such that $S=S_0\sqcup S_0^{-1}$.	
Let $\Theta$ be a ping-pong partition for $(G,S)\subset \Homeo_+(\T)$. After Lemma \ref{lem_collapse}, we can assume that $\Theta$ is strictly  Markovian, and after Lemma \ref{lem_geometric} we can assume that $\Theta$ is geometric.
Using Lemma \ref{l:interpolationCw}, for every $s\in S_0$ we can find a real-analytic diffeomorphism $\overline s\in \Diff_+^\omega(\T)$ such that $s(x)=\overline s(x)$ for every $x\in \Delta_0\cup s^{-1}(\Delta_0)$, and such that if $I\in\mathcal I$ is such that $|s(I)|>|I|$ then $\overline s'(x)>1$ on $I$, whereas if $|s^{-1}(I)|>|I|$, then $(\overline s^{-1})'(x)>1$ on $I$. The first condition guarantees that the marked subgroup $(\overline G=\langle \overline S\rangle,\overline S)$, where $\overline S=\{\overline s,\overline s^{-1}:s\in S_0\}$, has an equivalent ping-pong partition (it is basically the starting partition $\Theta$), which is moreover geometric. By Theorem \ref{t:bpvfreeB}, this gives that the action of $\overline G$ is semi-conjugate to the action of $G$. The second condition on derivatives guarantees that the partition $\Theta$ for $(\overline G,\overline S)$ is infinitesimally geometric, so that from Lemma \ref{l.minimality} we deduce that the action of $\overline G$ is minimal.
\end{proof}

\subsection{Classical and exotic examples}\label{s:exotic}

In this subsection we discuss various examples of ping-pong partitions. 
First we illustrate how one can produce examples of actions of some amalgamated free products or HNN extensions. The examples we give are quite basic, as considering more complex examples would lead to unreadable pictures.

\begin{ex}\label{ex:free2}
The standard example is the partition for the standard action of $\PSL(2,\Z)\cong \Z_2*\Z_3$, which corresponds to the classical Farey tessellation of the disc (see Example \ref{ex:Farey}).
The second example is the partition for the standard action of $\mathsf{SL}(2,\Z)\cong \Z_4*_{\Z_2}\Z_6$, which is a lift of degree 2 of the previous partition.
See Figure~\ref{fig:clex2} (left).
\end{ex}

\begin{ex}\label{ex:free3}
Similarly we can construct examples of arbitrary free products of finite cyclic group. An example for the free product $\Z_3*\Z_4*\Z_5$ appears in  Figure~\ref{fig:clex2} (right).
\end{ex}

\begin{ex}\label{ex:free4}
It is also easy to exhibit examples of ping-pong partitions for HNN extensions corresponding to free products. In Figure~\ref{fig:clex4} (left), we give an example for the group $\Z_3*_{\{\id\}}\cong \Z_3*\Z$.
\end{ex}

\begin{ex}\label{ex:free5}
Taking lifts of the previous examples, we can obtain ping-pong partitions for actions of HNN extensions. See Figure~\ref{fig:clex4} (right).
\end{ex}

\begin{figure}[ht]
\includegraphics{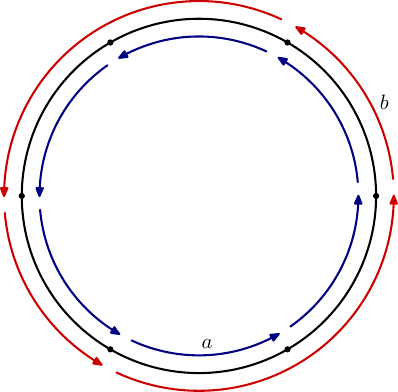}\quad\quad\includegraphics{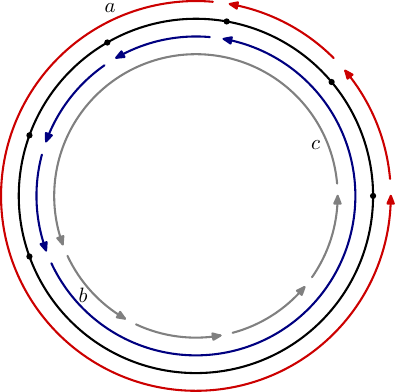}
\caption{Left: ping-pong partition for $\Z_4*_{\Z_2}\Z_6\cong \mathsf{SL}(2,\Z)$ (Example \ref{ex:free2}). Right: ping-pong partition for an action of $\Z_3*\Z_4*\Z_5$ (Example \ref{ex:free3}).}
\label{fig:clex2}
\end{figure}

\begin{figure}[ht]
\includegraphics{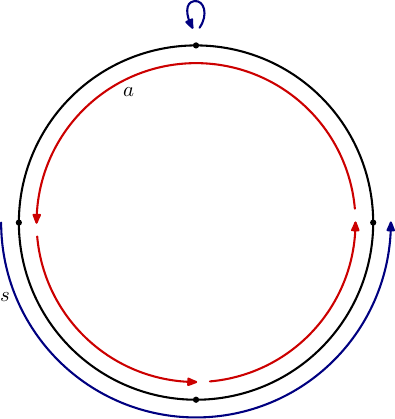}\quad\quad \includegraphics{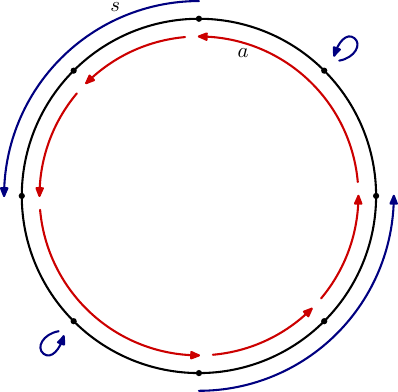}
\caption{Left: a ping-pong partition for an action of the HNN extension $\Z_3*_{\{\id\}}\cong \Z_3*\Z$ (Example \ref{ex:free4}). Right: a ping-pong partition for an action of the HNN extension $\Z_6*_{\Z_2}$ (Example \ref{ex:free5}).}
\label{fig:clex4}
\end{figure}

The next series of examples are for \emph{free groups}, and they illustrate how one can easily construct locally discrete subgroups of $\Diff_+^\omega(\T)$ whose action is minimal, and which are not conjugate into a central extension of $\PSL(2,\R)$.

\begin{ex}\label{ex:exotic1}
Using Theorem \ref{t.minimal_actions}, we can find elements $f,g,h\in \Diff_+^{\omega}(\T)$ that generate a locally discrete, free group of rank $3$ that acts minimally and such that: 
\begin{itemize}
	\item $f$ has exactly four fixed points: two hyperbolic attracting and two hyperbolic repelling;
	\item $g$ and $h$ have exactly two fixed points each, one hyperbolic attracting and one hyperbolic repelling.
\end{itemize}
See Figure~\ref{fig:ex1} (left).
\end{ex}

\begin{ex}\label{ex:exotic2}
Similarly, we can find $f,g\in \Diff_+^{\omega}(\Sc)$ that generate a locally discrete, free group of rank $2$ that acts minimally and such that:  
\begin{itemize}
	\item $f$ has exactly two fixed points, one hyperbolic attracting and one hyperbolic repelling;
	\item $g$ has exactly two fixed points, both of them parabolic.
\end{itemize}
See Figure~\ref{fig:ex1} (right).
\end{ex}

\begin{ex}\label{ex:exotic3}
Whereas the first two examples can be realized by choosing generators that individually belong to some finite central extension of $\PSL(2,\R)$, it is possible to describe other examples where this does not happen.

For instance, there exist $f,g,h\in \Diff_+^{\omega}(\T)$ that generate a locally discrete, free group of rank $3$ that acts minimally and such that: 
\begin{itemize}
	\item $f$ has exactly three fixed points: one parabolic, one hyperbolic attracting and one hyperbolic repelling;
	\item $g$ has exactly two fixed points, one hyperbolic attracting and one hyperbolic repelling;
	\item $h$ has exactly one fixed point, which is parabolic.
\end{itemize}
See Figure~\ref{fig:ex3}.
\end{ex}

\begin{figure}[ht]
\includegraphics{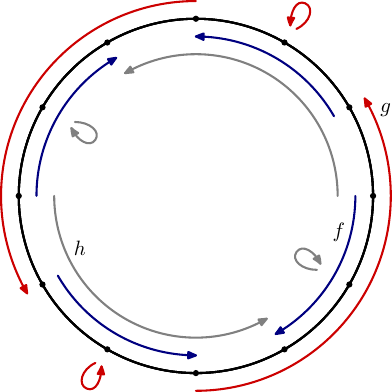}\quad\quad\includegraphics{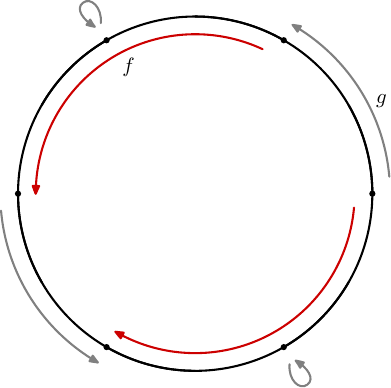}
\caption{Left: ping-pong partition  for Example \ref{ex:exotic1}. Right: ping-pong partition  for Example \ref{ex:exotic2}.}
\label{fig:ex1}
\end{figure}

\begin{figure}[ht]
\includegraphics{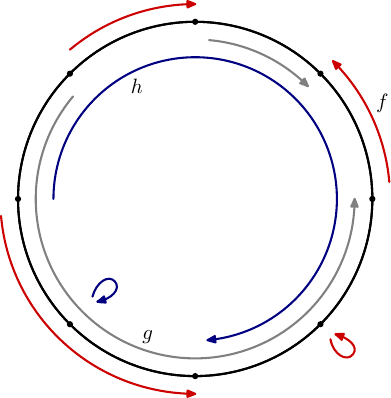}
\caption{Ping-pong partition  for Example \ref{ex:exotic3}.}
\label{fig:ex3}
\end{figure}

Note that, even if an action can be described by inequivalent ping-pong partitions, Example \ref{ex:exotic3} is indeed different from Example \ref{ex:exotic1}. In this case, this can be checked combinatorially by counting the different orbits of gaps of the partition, which is an invariant  for the action (see also the related work \cite{MMRT}). For the action in Example \ref{ex:exotic1} there are 6 such orbits, while the action in Example \ref{ex:exotic3} has 4 such orbits.


\section{Dynamical properties of DKN partitions}\label{s:dynamics}
\subsection{The DKN partition for free groups}\label{ss.Free_DKN}

In the case of locally discrete free subgroups $G\subset \Diff_+^\omega(\T)$, the ping-pong partitions given by Theorem \ref{t:bpvfreeA} are a straightforward consequence of the work of Deroin, Kleptsyn, and Navas \cite{DKN2014}. In this case a marking is simply a choice of a free, symmetric generating system $S$ for $G$. We will say that $(G,S)$ is a \emph{marked free group}. We will denote by $\|\cdot\|$ the word-length on $G$ with respect to the generating system $S$. 
For any $s\in S$, consider the subset $W_s$ of $G$ defined by
\begin{equation}\label{eq:Wsets}
	W_s=\left \{g\in G\mid g=s_\ell\cdots s_1\text{ in normal form, with } s_1=s\right \},
\end{equation}
and then consider the subset $U_s$ of the circle defined by
\begin{equation}\label{eq:Usets}
	U_{s}=\left\{x\in\T\,\middle\vert\,\exists\text{ neighbourhood }I_x\ni x\text{ s.t.~}\lim_{n\to\infty}\sup_{g\notin W_s,\,\|g\|\ge n}|g(I_x)|=0\right\}.
\end{equation}

\begin{rem}\label{r:differentUs}
	The definition \eqref{eq:Usets} of the subsets $U_s$ is the one given in \cite{MarkovPartitions1} and it slightly differs from the one used in \cite{DKN2014}. We refer to Appendix \ref{appendix} for a comparison between the two different definitions (see in particular Proposition \ref{p.equivalence_DKN}).
\end{rem}

The following properties of the collection of subsets $\{U_s\}_{s\in S}$ have essentially been proved in \cite{DKN2014}, see also \cite[Theorem 3.1]{MarkovPartitions1} and discussions therein. See also Appendix \ref{appendix}.

\begin{thm}[Deroin, Kleptsyn, and Navas]\label{t:DKNfree}
	Let $(G,S)\subset \Diff_+^\omega(\T)$ be a finitely generated, locally discrete, marked free group of real-analytic circle diffeomorphisms, with minimal invariant set $\Lambda$.
	Consider the collection $\{U_s\}_{s\in S}$ defined in \eqref{eq:Usets}. We have:
	\begin{enumerate}[\rm (1)]
		\item \label{i:open_free} every subset $U_s$  is open;
		
		\item \label{i:finite_free} every subset $U_s$ has finitely many connected components;
		
		\item \label{i:intersection_on_minimal} any two different subsets  $U_s$ have empty intersection inside the minimal invariant set $\Lambda$;
		
		\item\label{i:cover} the union of the subsets $U_s$ covers all but finitely many points of $\Lambda$;
		
		\item \label{eq_ping_pong_free} if $s,t\in S$, are such that $t\neq s$ then  
		$s(U_t)\subset U_{s^{-1}}$.
	\end{enumerate}
\end{thm}

Note that Theorem  \ref{t:DKNfree} does not imply that the collection of subsets $\{U_s\}_{s\in S}$ is a ping-pong partition, in the sense of Definition \ref{d:markov-partition}, as two distinct elements of $\{U_s\}_{s\in S}$ may have non-empty intersection outside the minimal invariant compact set $\Lambda\subset \T$.
For this reason, for every $s\in S$ we introduce a subset $\hat U_s\subset U_s$ defined by the following procedure: if an endpoint of a connected component of $U_s$ belongs to a gap $J\subset \T\setminus \Lambda$ (that is, a connected component of the complement $\T\setminus \Lambda$), we remove $U_s\cap \overline J$ from $U_s$. As $U_s$ has finitely many connected components, we only have to do this operation finitely many times.
Keeping the terminology of \cite{MarkovPartitions1}, we refer to the collection of subsets $\{U_s\}_{s\in S}$ as to the \emph{DKN partition}, whereas $\{\hat U_s\}_{s\in S}$ is the \emph{DKN ping-pong partition} for $(G,S)$ (which is a ping-pong partition in the sense of Definition \ref{d:markov-partition}).

The following notion plays a crucial r\^ole in the proof of Theorem \ref{t:DKNfree}, and it will also be useful here.

\begin{dfn}\label{d.wandering}
	Let $(G,S)\subset \Diff_+^\omega(\T)$ be a finitely generated, locally discrete, marked free group of real-analytic circle diffeomorphisms, and let $\{U_s\}_{s\in S}$ be the associated DKN partition. Let $s\in S$ and let $I$ be a connected component of $U_s$. 
	A non-trivial element $g\in G$, written in normal form as $g=s_\ell\cdots s_1$, is \emph{wandering} (with respect to $I$) if the following conditions are satisfied:
	\begin{enumerate}
		\item $s_1\neq s$, or equivalently $g\notin W_s$;
		\item the intermediate images $s_k\cdots s_1(I)$, for $k\in\{0,\ldots,\ell-1\}$, are pairwise disjoint.
	\end{enumerate}
	Moreover, if $g$ also satisfies $g(I)\cap I\neq\varnothing$, then we say that $g$ is a \emph{first return} (to $I$).
\end{dfn}

\begin{rem}
	Wandering elements are called admissible elements in the terminology in \cite{DKN2014}.
\end{rem}

The following facts have been proved in \cite[\S 3.4]{DKN2014} in the case of minimal actions, and in \cite[\S 4]{DKN2014} for actions with minimal invariant Cantor sets. Formally speaking (see Remark \ref{r:differentUs}), the results in \cite{DKN2014} apply to connected components of the partition $\{\tcM_s\}_{s\in S}$, but the proofs can be easily adapted to the partition $\{U_s\}_{s\in S}$. Otherwise, the reader can see the discussion in Appendix \ref{appendix}, from which it is possible to deduce all desired results for the partition $\{U_s\}_{s\in S}$ from those for the partition $\{\tcM_s\}_{s\in S}$.

\begin{lem}\label{l:first return prop}
	Let $(G,S)\subset \Diff_+^\omega(\T)$ be a finitely generated, locally discrete, marked free group of real-analytic circle diffeomorphisms, and let $\{U_s\}_{s\in S}$ be the associated DKN partition. For any $s\in S$ and  connected component $I=(x_-,x_+)$ of $U_s$,  the following properties hold:
	\begin{enumerate}[\rm (1)]
		\item two distinct wandering elements $g$ and $h$ have disjoint images: $g(I)\cap h(I)=\varnothing$;
		\item there exist two first returns $g_-$ and $g_+$ fixing $x_-$ and $x_+$ respectively, such that $g^k_\pm(x)\to x_\pm$ as $k\to\infty$ for every $x\in I$;
		\item there exists a non-empty closed interval $J\subset I$ such that 
		$I=g_-(I)\sqcup J\sqcup g_+(I)$;
		in particular for every first return $g$ which is distinct from $g_{\pm}$ one has $g(I)\subset J$;
		\item \label{decomposition} for every $g\notin W_s$, there exists $n\in\N$ such that $g$ can be written uniquely as
		\begin{equation*}\label{eq.decomposition_first_returns}
			g=\lambda g_n\cdots g_1,
		\end{equation*}
		where for every $i\in\{1,\ldots,n\}$, the element $g_i\in G$ is a first return and the element $\lambda\in G$ is either wandering or trivial.
	\end{enumerate}
\end{lem}


From this we deduce the following proposition.

\begin{prop}\label{p.I_tend_zero}
	Let $(G,S)\subset \Diff_+^\omega(\T)$ be a finitely generated, locally discrete, marked free group of real-analytic circle diffeomorphisms, and let $\{U_s\}_{s\in S}$ be the associated DKN partition. For any $s\in S$ and  connected component $I$ of $U_s$, one has
	\[\lim_{n\to\infty}\sup_{g\notin W_s,\,\|g\|\ge n}|g(I)|=0.\]
\end{prop}

\begin{proof}
	We use Lemma \ref{l:first return prop} to fix some notation and constants.
	First, for every element $g\notin W_s$ take the largest $k=k(g)\in \N$ such that
	$g=\tilde g g_0^k$, with $g_0\in \{g_-,g_+\}$ and $\tilde g\notin W_s$. If $\widetilde g$ is not wandering, there exist a first return $g_1\neq g_0$, and $\overline g\notin W_s$ such that $g=\overline g g_1g_0^k$.
	Write 	$I=g_-(I)\sqcup J\sqcup g_+(I)$ as in Lemma \ref{l:first return prop}, then the subset
	\[J':=g_-(g_+(I))\cup J\cup g_+(g_-(I))\]
	is a closed subinterval of $I$.
	Fix now $\eps >0$. By compactness and the definition \eqref{eq:Usets} of $U_s$, there exists $N_0\in\N$ such that
	\begin{equation}\label{eq:first_return0}
		|g(J')|\leq\eps \quad \text{for every $g\notin W_s$ with $\|g\|\ge N_0$}.
	\end{equation}
	Since for fixed $N\in\N$, there are only finitely many elements $g\in G$ with $\|g\|\leq N$, and since they are uniformly continuous on $\T$,  we can fix $\delta(N)\in (0,\eps]$ such that 
	\begin{equation}\label{eq:first_return1}
		|g(I')|\leq\eps\quad \text{for every $g\in G$ with $\|g\|\leq N$ and every interval $I'$ with $|I'|\leq\delta(N)$}.
	\end{equation}
	Observe also that from Lemma \ref{l:first return prop}, we have 
	\begin{equation}\label{eq:admissible}
		\lim_{n\to \infty}\sup_{h\text{ wandering},\,\|h\|\ge n} |h(I)|=0.
	\end{equation}
	In particular there exists a positive integer $N_1\in\N$ such that
	\begin{equation}\label{eq:first_return2}
		|h(I)|\leq\delta(N_0)\quad \text{for every wandering element $h$ of length $\|h\|\geq N_1$}.
	\end{equation}
	Finally if $g_0=g_\pm$ we have $|g_0^k(I)|\to 0$ as $k\to\infty$, so there exists $N_2>0$ such that
	\begin{equation}\label{eq:first_return3}
		|g_0^k(I)|\leq\delta(N_0+N_1)\quad \text{for every $k\geq N_2$}.
	\end{equation}
	Fix now $N:=N_0+N_1+N_2\max\{\|g_-\|,\|g_+\|\}$, take an element $g\notin W_s$ with $\|g\|\ge N$, and write $g=\widetilde gg_0^k$ as at the beginning of the proof. 
	
	\setcounter{case}{0}
	
	\begin{case}\label{case1_first_return}
		$\widetilde g$ is wandering.
	\end{case}
	\begin{proof}[Proof in Case \ref{case1_first_return}]
		Remark that we must have  $\|\widetilde g\|\ge N_0+N_1$ or $k\ge N_2$.	
		Assume first that $\|\widetilde g\|\ge  N_0+N_1$, so that by \eqref{eq:first_return2} we have 
		\[|\widetilde gg_0^k(I)|\le |\widetilde g(I)|\le \delta(N_0)\le \eps.\]
		If otherwise $\|\widetilde g\|\le N_0+N_1$, then $k\ge N_2$ and we deduce from \eqref{eq:first_return3} that $|g_0^k(I)|\le \delta(N_0+N_1)$. As $\|\widetilde g\|\le N_0+N_1$, then $|\widetilde gg_0^k(I)|\le \eps$ by \eqref{eq:first_return1}.
	\end{proof}
	
	\begin{case}\label{case2_first_return}
		$\widetilde g$ is not wandering.
	\end{case}
	\begin{proof}[Proof in Case \ref{case2_first_return}]
		We write $g=\overline g g_1g_0^k$ as at the beginning of the proof.
		Remark that we must have  $\|\overline g\|\ge N_0$ or $\|g_1\|\ge N_1$ or $k\ge N_2$.
		Assume first $\|\overline g\|\ge N_0$, and note that as $g_1\neq g_0$ then we have $g_1g_0^k(I)\subset J'$. Hence $|\overline gg_1g_0^k(I)|\le |\overline g(J)|\le \eps$ by \eqref{eq:first_return0}. Next, assume $\|\overline g\|< N_0$ and $\|g_1\|\ge  N_1$, and note that by \eqref{eq:first_return2} we have
		\[|g_1g_0^k(I)|\le |g_1(I)|\le \delta(N_0).\]
		As $\|\overline g\|< N_0$,  we conclude by \eqref{eq:first_return1} that
		$|\overline gg_1g_0^k(I)|\le \eps$. Finally, the case where $\|\widetilde g\|\le N_0+N_1$ and $k\ge N_2$ follows from \eqref{eq:first_return3} and \eqref{eq:first_return1} as in Case \ref{case1_first_return}.
	\end{proof}
	
	The list of cases being exhaustive, this proves that if $g\notin W_s$ is such that $\|g\|\ge N$, then $|g(I)|\le \eps$, as desired.
\end{proof}

\subsection{Multiconvergence property}
In this subsection we prove Theorem \ref{mthm:multiconvergence} and its consequences.
We first observe that it is enough to consider the case of free groups.

\begin{lem}\label{l:enough_free}
	Let $G\subset \Homeo_+(\T)$ be a group of circle homeomorphisms, and let $H\subset G$ be a  subgroup of finite index. The following statements are equivalent:
	\begin{enumerate}[\rm (1)]
		\item the group $G$ has the multiconvergence property;\label{i:multG}
		\item the subgroup $H$ has the multiconvergence property.\label{i:multH}
	\end{enumerate}
\end{lem}

\begin{proof}
	The only non-trivial implication is \eqref{i:multH} $\Rightarrow$ \eqref{i:multG}. That is, we assume that $H$ satisfies the multiconvergence property with a uniform constant $K\in\N$.	
	Let $\{g_n\}_n\subset G$ be an infinite sequence of distinct elements. As $H\subset G$ has finite index, there exists a coset $Ht$ to which infinitely many elements of the sequence $\{g_n\}_n$ belong. By finite-index assumption, upon extracting, we can assume that $\{g_n\}_n$ is contained in a single coset $Ht$, and we write $g_n=h_nt$, with $h_n\in H$.
	Clearly the sequence $\{h_n\}_n\subset H$ consists of distinct elements, and up to passing to a subsequence, we can assume that there exist finite subsets $A$ and $R\subset \T$ with $\#A=\#R\le K$ for which the sequence of restrictions $\{h_n\restriction_{\T\setminus R}\}_n$ converges to the locally constant map $h_\infty$, with discontinuities at $R$ and values in $A$, as in Definition~\ref{d:multiconvergence}.	Then the sequence of restrictions $\{g_n\restriction_{\T\setminus t^{-1}(R)}\}_n$ converges to the map $h_\infty t$, which is locally constant on $\T\setminus t^{-1}(R)$, and whose image is $A$. Therefore $G$ satisfies the multiconvergence property, with the same uniform bound $K\in\N$ as for the subgroup $H$.
\end{proof}

\begin{proof}[Proof of Theorem~\ref{mthm:multiconvergence}]
	After Lemma~\ref{l:enough_free}, it is enough to assume that $G$ is a free group, freely generated by a finite symmetric set $S$; we denote by $||\,\cdot\,||$ the corresponding word-length function. Consider the corresponding DKN partition, given by the subsets $U_s$ as in \eqref{eq:Usets}. We will prove that $G$ verifies the multiconvergence property with $K=\max_{s\in S}b_0(U_s)$, where  $b_0(U_s)$ stands for the zeroth Betti number (number of connected components) of the open subset $U_s$. For this, let us take an infinite sequence $\{g_m\}_m$ of distinct elements of $G$.
	Up to passing to a subsequence if necessary, we can assume that $\{g_m\}_m$ converges monotonically to a point of the boundary $\partial G$ of the free group $G$: there exist a monotone sequence of integers $\{n_m\}_m\subset \N$ and a sequence $\{s_{n}\}_n\subset S$ of generators such that for every $n\in\N$, $s_{n}\neq s_{n+1}^{-1}$ and for every $m\in \N$ one has
	\[g_m= s_{n_m} s_{{n_m}-1}\cdots s_{1}.\]
	First, we find the repelling subset $R$. We start by the following elementary observation.
	
	\begin{claim1}\label{cl:easy}
		For every $n\in\N$ we have
		\begin{equation}
			\label{eq.easy_claim}
			s_{1}^{-1}\cdots s_{n}^{-1} (U_{s_{{n+1}}})\subset s_{1}^{-1}\cdots s_{{n-1}}^{-1} (U_{s_{{n}}}).
		\end{equation}
	\end{claim1}
	\begin{proof}[Proof of claim]
		Indeed for every $n\in\N$ we have $s_{{n+1}}\neq s_{{n}}^{-1}$ so by  Theorem \ref{t:DKNfree}.\eqref{eq_ping_pong_free} we have $s_{n}^{-1}(U_{s_{{n+1}}})\subset U_{s_{n}}.$ So the claim follows by precomposing by the element $s_{1}^{-1}\cdots s_{{n-1}}^{-1}.$
	\end{proof}
	
	As a consequence of Claim \ref{cl:easy}, the subsets $R_n:=\overline{s_1^{-1}\cdots s_{{n-1}}^{-1} (U_{s_{{n}}})}$
	form a decreasing sequence of non-empty compact subsets in $\overline{U_{s_1}}$. So the intersection
	\[R:=\bigcap_{n\in\N} R_n\]
	is a non-empty compact subset.
	
	\begin{claim1}\label{cl:cc}
		The number of connected components of $R$ is bounded above by $K=\max_{s\in S}b_0(U_{s})$.
	\end{claim1}
	
	\begin{proof}[Proof of claim]
		Indeed for every $n\in\N$, the subset $R_n$ has the same number of connected components as the subset $U_{s_{n}}$, which is less than or equal to $K$ by our choice of $K$.
	\end{proof}
	
	\begin{claim1}
		The subset $R\subset \T$ is finite, and thus by Claim \ref{cl:cc} it has at most $K$ points.
	\end{claim1}
	
	\begin{proof}[Proof of claim]
		A connected component $I$ of $R$ is the decreasing intersection of connected components $I_n$ of $R_n$. We must prove that $|I_n|\to 0$ as $n\to\infty$.
		
		By the pigeonhole principle, there exist $s\in S$, a connected component $J$ of $U_s$ and a sequence $\{m_k\}_k\subset\N$ such that for every $k\in\N$, $s_{m_k}=s$ and 
		\[I_{m_k}=\overline{s_1^{-1}\cdots s_{m_k-1}^{-1}(J)}.\]
		Now since by hypothesis $s_{m_k-1}^{-1}\neq s_{m_k}$, it follows from  Proposition \ref{p.I_tend_zero} that $\lim_{k\to\infty}|I_{m_k}|=0$. As the intervals $I_n$ are nested, this gives the desired conclusion.
	\end{proof}
	
	We now turn to the construction of the attracting subset $A$. Note that by Claim \ref{cl:cc}, there exists $N_0\in\N$ such that the number of connected components of $R_n$ does not depend on $n\ge N_0$ (this number equals $\#R$). Without loss of generality, we can assume that the sequence of elements $\{g_m\}_m$ is such that $n_m\ge N_0$ for all $m\in \N$.

	By construction and Theorem \ref{t:DKNfree}.\eqref{eq_ping_pong_free}, for every $m\in\N$, the image $g_m(\T\setminus R_{N_0})$ is contained in $U_{s_{n_m}^{-1}}$
	and has $\# R$ connected components. Let $J_0$ be a connected component of $g_0(\T\setminus R_{N_0})$, let $I$ be the connected component of $U_{s_{n_0}^{-1}}$ containing it and let $J_m$ denote $g_mg_0^{-1}(J_0)\subset U_{s_m^{-1}}$ for every $m\in\N$. By Proposition \ref{p.I_tend_zero}, we have $|J_m|\leq|g_mg_0^{-1}(I)|\to 0$ as $m\to\infty$. Let $x_m$ denote the midpoint of $J_m$. By compactness of $\T$, upon passing to a subsequence if necessary, we may assume that $x_m$ converges to a point $a\in\T$.
	Running the argument for every connected component of $g_0(\T\setminus R_{N_0})$, and upon considering a subsequence, we can also assume that there exists a finite subset $A$ such that for every $y\in \T\setminus R_{N_0}$, we have $g_m(y)\to a$ as $m\to\infty$ for some $a\in A$. By construction $\#A\leq\#R$.

	Now for $n\geq N_0$ and $m$ such that $n_m\geq n$ the subset $g_m(\T\setminus R_n)$ is contained in $U_{s_m^{-1}}$ and each of its connected components contains a unique connected component of $g_m(\T\setminus R_{N_0})$.
	Using the argument above, we see that the length of such a component tends to $0$ as $m$ tends to infinity so for every $y\in\T\setminus R_{n}$, we have $g_m(y)\to a$ as $m\to\infty$ for some $a\in A$. Finally, since $n$ is arbitrary, this fact holds for every $y\in\T\setminus R$. This proves the multiconvergence property.
\end{proof}

\begin{rem}
	We observe that if $G\subset\Homeo_+(\T)$ has the multiconvergence property, then it is always possible to choose (with the notation as in Definition \ref{d:multiconvergence}) the subsets $A$ and $R$ such that  $\#A=\#R$. Indeed, if we have $\#A<\#R$, then there exist a point $r\in R$ and a point $a\in A$ such that if $J_-$ and $J_+$ are the two connected components of $\T\setminus R$ adjacent to $r$, we have $g_{n_k}(y)\to a$ as $k\to\infty$ for every $y\in J_-\cup J_+$. In particular we also have $g_{n_k}(r)\to a$ as $k\to\infty$ and we may remove $r$ from $R$.
\end{rem}

\begin{lem}\label{l.bounded_fixed}
	Let $G\subset \Homeo_+(\T)$ be a subgroup with the multiconvergence property. Then the number of fixed points of a non-trivial element in $G$ is uniformly bounded.
\end{lem}

\begin{proof}
	Let $K\in\N$ be the constant from Definition \ref{d:multiconvergence} of multiconvergence property.
	Assume that $g\in G$ is an element of infinite order and consider the sequence $\{g^k\}_{k\in \N}$. As a consequence of the multiconvergence property the cardinality of any finite invariant cannot exceed $\#R=\#A\leq K$. This means in particular that $g$ has at most $2K$ fixed points. 
\end{proof}

\begin{proof}[Proof of Corollary \ref{c:bound_period}]
	Direct consequence of Theorem \ref{mthm:multiconvergence} and Lemma \ref{l.bounded_fixed}.	
\end{proof}

\begin{proof}[Proof of Corollary \ref{cor:rot_spectr_charac}] As we mentioned in Section \ref{ss:multicon_intro}, Matsuda proved in \cite{Matsuda2009} that subgroups of $\Diff^{\omega}_+(\T)$ with finite rotation spectrum are locally discrete. The converse implication follows from Corollary~\ref{c:bound_period} and  the fact that finitely generated virtually free groups have bounded torsion. We propose an alternative and simpler argument to prove this direction (which does not rely on the multiconvergence property, but on Theorem \ref{t:DKNfree} only).
	
	First, it is enough to consider free groups since if $G$ is virtually free, it contains a normal free subgroup $H$ of finite index $k$ (by the so-called Poincar\'e lemma). Hence for every $g\in G$, we have $g^k\in H$ and $k\rot(g)=\rot(g^k)\in\rot(H)$. So if the rotation spectrum $\rot(H)$ is finite, then so is $\rot(G)$ (in fact, a finite-index normal subgroup $H\lhd G$ can be obtained from the function rotation number, see the proof of \cite[Theorem C]{MarkovPartitions1}).
	
	Assume now that $G$ is free and locally discrete, and choose a free, symmetric generating system $S$, with associated DKN partition $\{U_s\}_{s\in S}$. Let $g\in G$ be a non-trivial element. As rotation number is a conjugacy-invariant, we can assume that $g$ can be written in a cyclically reduced normal form $g=s_n\cdots s_1$, for some $n\ge 1$ (that is, we assume $s_n\neq s_1^{-1}$). By Theorem \ref{t:DKNfree}, we have
	\[g(U_{s_n^{-1}})=s_n\cdots s_1(U_{s_n^{-1}})\subset U_{s_n^{-1}}.\]
	By Theorem \ref{t:DKNfree}, the number $q$ of connected components of $U_{s_n}$ is finite. So $g^q$ preserves all connected components of $U_{s_n}$ and therefore admits a fixed point: $\rot(g^q)=0$. Since $q$ is uniformly bounded by $K=\max_{s\in S}b_0(U_s)$, we deduce that $\rot(G)$ is finite, concluding the proof.
\end{proof}

\appendix

\section{Comparison of partitions} \label{appendix}

As we mentioned in Remark \ref{r:differentUs}, in \cite{DKN2014} the definition of the sets $U_s$ (appearing there as $\widetilde{\mathcal{M}}_\gamma$) is not exactly the same. Since the present paper relies on the results of \cite{DKN2014}, we discuss here differences and similarities of the two definitions.
As in Section \ref{ss.Free_DKN}, we consider a locally discrete, marked free group $(G,S)$ of real-analytic circle diffeomorphisms, with minimal invariant set $\Lambda$. The subsets $W_s\subset G$ and $U_s\subset \T$ will be the same as those defined at \eqref{eq:Wsets} and \eqref{eq:Usets}, respectively.
We recall now the notation  from \cite[\S 3.4]{DKN2014}. For $g\in G$ written in normal form as $g=s_\ell\cdots s_1$, and $x\in\T$, set

\begin{equation}\label{eq:Sintermediate}
\widehat{S}(g,x):=\sum_{k=0}^{\ell-1} (s_\ell\cdots s_1)'(x),
\end{equation}
and, for a generator $s\in S$, set
\begin{equation}\label{eq:Msets}
\tcM_s:=\left\{x\in\T\,\middle\vert\,\sup_{g\notin W_s} \widehat{S}(g,x)<\infty\right\}.
\end{equation}
Note that as soon as one $\tcM_s$ is non-empty (which is one of the most difficult results established in \cite{DKN2014}), then every other $\tcM_{s'}$ is non-empty, and their union is dense in $\Lambda$ (see \cite[Lemma 3.20]{DKN2014}). Moreover, for any $s\in S$, only finitely many connected components of $\tcM_s$ intersect $\Lambda$ (this is \cite[Lemma 3.30]{DKN2014} in case of minimal actions, and \cite[Lemma 4.7]{DKN2014} for the case of actions with minimal invariant Cantor set). It follows that the union $\bigcup_{s\in S}\tcM_s$ covers the whole minimal invariant set $\Lambda$, but finitely many points.

\begin{lem}\label{l:inclusion_appendix}
Let $(G,S)$ be a locally discrete, marked free group of real-analytic circle diffeomorphisms. For every generator $s\in S$, we have the inclusion $\tcM_s\subset U_s$.
\end{lem}

\begin{proof}
The first part follows directly from the usual control of distortion (namely Schwarz's lemma: see \cite[Proposition 2.4]{DKN2014}). Let $x\in\tcM_s$ and $C_x=\sup_{g\notin W_s}\widehat{S}(g,x)<\infty$. Then there exists a neighbourhood $I_x$ of $x$ whose size depends only on the generators of $G$ and on $C_x$ such that for all $g\notin W_s$, and $y,z\in I_x$, one has $\frac{g'(y)}{g'(z)}\leq 2.$ By definition, if we restrict ourselves to elements $g\notin W_s$, we have $\lim_{||g||\to\infty}g'(x)=0,$
so the uniform distortion control implies that $\lim_{||g||\to\infty}|g(I_x)|=0$. We deduce that $x\in U_s$.
\end{proof}

Note that this implies that, for every $s\in S$, only finitely many connected components of $U_s$ intersect $\Lambda$, and that $\bigcup_{s\in S}U_s$ covers the whole minimal invariant set $\Lambda$ but finitely many points.

We also deduce the following consequence, which is the analogue of Proposition \ref{p.I_tend_zero}. Observe that the proof of Proposition \ref{p.I_tend_zero} relies on Lemma \ref{l:first return prop}, which is stated for the DKN partition $\{U_s\}_{s\in S}$ but formally only proved in \cite{DKN2014} for the partition $\{\tcM_s\}_{s\in S}$ (although the proof can be easily adapted to the partition $\{U_s\}_{s\in S}$).

\begin{prop}\label{l:first return appendix}
	Let $(G,S)$ be a locally discrete, marked free group of real-analytic circle diffeomorphisms. For every generator $s\in S$ and connected component $I$ of $\tcM_s$, one has
	\[
	\lim_{n\to \infty}\sup_{g\notin W_s,\,\|g\|\ge n}|g(I)|=0.
	\]
\end{prop}

Next, we give a more precise description of the inclusion $\tcM_s\subset U_s$.

\begin{lem}\label{l_concomp_same_Us}
	Let $(G,S)$ be a locally discrete, marked free group of real-analytic circle diffeomorphisms with minimal invariant set $\Lambda$. Fix a generator $s\in S$ and two distinct connected components $I$ and $I'$ of $\tcM_s$.
\begin{enumerate}[\rm (1)]
\item\label{i:concomp appendix} If both closures $\overline I$ and $\overline{I'}$ intersect the closure of a gap of $\Lambda$, then they are contained in distinct connected components of $U_s$.
\item\label{ii:concomp appendix} If $I$ and $I'$ share a common endpoint, then they are contained in the same connected component of $U_s$. 
\end{enumerate}
\end{lem}

\begin{rem}
	Note that, as a consequence of \eqref{i:concomp appendix}, in the situation \eqref{ii:concomp appendix} the common endpoint cannot belong to the closure of a gap of $\Lambda$
\end{rem}

\begin{proof}
We first prove the assertion \eqref{i:concomp appendix}. Assume that both closures $\overline I$ and $\overline{I'}$ intersect the closure of the same gap $J$ of $\Lambda$. Assume that $I$ and $I'$ are contained in the same connected component of $U_s$, so the closure  $\overline{J}$ of the gap is inside this component. By compactness and the definition \eqref{eq:Usets} of $U_s$, this implies that
\begin{equation}\label{eq:image gap appendix}
\lim_{n\to \infty}\sup_{g\notin W_s,\,\|g\|\ge n}|g(J)|=0.
\end{equation}
However, using the analogue of Lemma \ref{l:first return prop} for the partition $\{\tcM_s\}_{s\in S}$, applied to $I$, we get a first return $g\notin W_s$ which fixes the endpoint of $I$ contained in $J$ (see more specifically \cite[Lemma 4.4]{DKN2014}). It follows that $g$ fixes the whole gap $J$, contradicting \eqref{eq:image gap appendix}.
Hence, the gap $J$ is not contained in a connected component of $U_s$,  and so the two connected components of $U_s$ containing $I$ and $I'$ respectively are disjoint.

Assume next that $I$ and $I'$ have a common endpoint $x$, and write $I_x=I\cup \{x\}\cup I'$, which is a neighbourhood of $x$. Note that for every $g\in G$, we have $|g(I_x)|=|g(I)|+|g(I')|$.
From Proposition \ref{l:first return appendix} applied to $I$ and $I'$ individually, we deduce that
\[
\lim_{n\to \infty}\sup_{g\notin W_s,\,\|g\|\ge n}|g(I_x)|=0
\]
and therefore $x\in U_s$. This proves \eqref{ii:concomp appendix}.
\end{proof}


Finally, we recall the following fact and its proof, appearing in \cite[Lemma 7.6]{MarkovPartitions1}.

\begin{lem}\label{l.UM_vidouille}
	Let $(G,S)$ be a locally discrete, marked free group of real-analytic circle diffeomorphisms. For every two distinct generators $s$ and $s'\in S$, we have $U_s\cap U_{s'}\cap\Lambda=\varnothing$.
\end{lem}

\begin{proof}
	Take distinct generators $s$ and $s'\in S$, and assume there exists $x\in U_s\cap U_{s'}\cap\Lambda$. Then there exists a neighbourhood $I_x$ of $x$ such that $\lim_{||g||\to\infty}|g(I_x)|=0$ (we use that for all $g\in G$, we have $g\notin W_s$ or $g\notin W_{s'}$ since $s\neq s'$). We have $\Lambda\cap I_x\neq\varnothing$ so Sacksteder's theorem implies the existence of a point $y\in\Lambda\cap I_x$ and an element $h\in G$ such that $h'(y)>1$. This contradicts that $\lim_{n\to\infty}|h^n(I_x)|=0$.
\end{proof}

It is now quite straightforward to deduce from the previous lemmas that the two partitions $\{U_s\}_{s\in S}$ and $\{\tcM_s\}_{s\in S}$ are ``equivalent'' in the following sense.

\begin{prop}\label{p.equivalence_DKN}
Let $(G,S)$ be a locally discrete, marked free group of real-analytic circle diffeomorphisms. For every generator $s\in S$, we have the inclusion $\tcM_s\subset U_s$ and the complement $U_s\setminus\tcM_s$ consists of intervals entirely contained in closures of gaps of $\Lambda$ and a finite number of points in $\Lambda$, which are topologically hyperbolic fixed points $x$ for some element $g$ in the group, but with derivative $g' (x) = 1$.
\end{prop}

\begin{proof}
	Fix a generator $s\in S$. Lemma \ref{l:inclusion_appendix} gives the inclusion $U_s\subset\tcM_s$. Let $C$ be a connected component of $U_s\setminus\tcM_s$. Assume first that $C=\{x\}$ is a single point, then Lemma \ref{l_concomp_same_Us}.\eqref{ii:concomp appendix} gives that this does not belong to the closure of a gap of $\Lambda$, and moreover, by the analogue of Lemma \ref{l:first return prop} for $\{\tcM_s\}_{s\in S}$ (see more specifically \cite[Lemma 3.23]{DKN2014}  for the case of minimal action, and \cite[Lemma 4.4]{DKN2014} for the case of minimal invariant Cantor set) its proof shows that there exists an element $g\notin W_s$ fixing $x$. As $x\in U_s$, the point $x$ must be a topologically hyperbolic fixed point for $g$, and moreover $g'(x)=1$ (this is given by \cite[Lemma 3.29]{DKN2014} in the case of minimal action, or by \cite[Lemma 4.6]{DKN2014} in the case of minimal invariant Cantor set).
	
	Assume next that $C$ is a non-trivial interval, and that $C$ contains a point $x$ of $\Lambda$ in its interior. Then $x$ is an accumulation point for $C\cap \Lambda$. As $\bigcup_{s'\in S}\tcM_{s'}$ covers $\Lambda$ with the exception of finitely many points, there must be a different generator $s'\neq s$ in $S$ such that $C\cap \tcM_{s'}\cap \Lambda\neq \varnothing$. Using Lemma \ref{l:inclusion_appendix} we deduce that $C\cap U_{s'}\cap \Lambda\neq \varnothing$, and since $C\subset U_s$, we then have $U_s\cap U_{s'}\cap \Lambda\neq \varnothing$, contradicting Lemma \ref{l.UM_vidouille}.
\end{proof}

\section*{Acknowledgments}
The results of this work and its companion \cite{MarkovPartitions1} were first announced in 2016. 
We warmly thank the institutions which have welcomed us during these years: the institutions of Rio de Janeiro PUC, IMPA, UFF, UFRJ, UERJ, the Universidad de la Rep\'ublica of Montevideo, and the Institut de Math\'ematiques de Bourgogne of Dijon.

\footnotesize{
	S.A. was partially supported by ANII via the Fondo Clemente Estable (project FCE\_3\_2018\_1\_148740), by CSIC, via the project I+D 389, by the MathAmSud project RGSD  (Rigidity and Geometric Structures in Dynamics) 19-MATH-04, by the LIA-IFUM and by the Distinguished Professor Fellowship of FSMP. He also acknowledges the project  ``Jeunes
	Géomètres'' of F. Labourie (financed by the Louis D. Foundation).
	
P.G.B. was supported by Ministerio de Educación, Cultura y Deporte (Grant No.\ MTM2017-87697-P) and Conselho Nacional de Desenvolvimento Científico e Tecnológico (CNPq).

D.F. was partially supported by the RFBR project 20-01-00420.

	V.K. was partially supported by the R\'eseau France-Br\'esil en Math\'ematiques, by the project ANR Gromeov (ANR-19-CE40-0007), and by Centre Henri Lebesgue
	(ANR-11-LABX-0020-01).
	
	C.M. was partially supported by CAPES postdoc program, Brazil (INCTMat \& PNPD 2015--2016), the
	Ministerio de Economía y Competitividad and Ministerio de Ciencia e Innovación, Grants MTM2014-
	56950--P and PID2020-114474GB-I00 (Spain UE), the Programa Cientista do Estado do Rio de Janeiro, FAPERJ, Brazil (2015--2018), CNPq research grant 310915/2019-8 (Brazil) and MathAmSud: ``Rigidity and Geometric Structures in Dynamics'' 2019-2020 (CAPES-Brazil). He also received support from the CSIC, via the Programa de Movilidad e Intercambios Acad\'emicos.

	M.T. was partially supported by PEPS -- Jeunes Chercheur-e-s -- 2017 and 2019 (CNRS), the R\'eseau France-Br\'esil en Math\'ematiques, MathAmSud RGSD  (Rigidity and Geometric Structures in Dynamics) 19-MATH-04,
	the project ANER Agroupes (AAP 2019 R\'egion Bourgogne--Franche--Comt\'e), and by the project ANR Gromeov (ANR-19-CE40-0007). His host department IMB receives support from the
	EIPHI Graduate School (ANR-17-EURE-0002).
}

\bibliographystyle{plain}
\bibliography{biblio2}

\begin{flushleft}

	{\scshape S\'ebastien Alvarez}\\
	CMAT, Facultad de Ciencias, Universidad de la Rep\'ublica\\
	Igua 4225 esq. Mataojo. Montevideo, Uruguay.\\
	email: \texttt{salvarez@cmat.edu.uy}

	\smallskip
	
	{\scshape Pablo G. Barrientos}\\
	Universidade Federal Fluminense\\
	Rua Prof. Marcos Waldemar de Freitas Reis, S/N -- Bloco H, 4o Andar\\
	Campus do Gragoatá, Niterói, Rio de Janeiro 24210-201, Brasil\\
	email:  \texttt{pgbarrientos@id.uff.br}
	
	\smallskip
	
	{\scshape Dmitry Filimonov}\\
	HSE University\\
	20 Myasnitskaya ulitsa,
	101000 Moscow, Russia\\
	email: \texttt{dfilimonov@hse.ru}
	
	\smallskip
	
	{\scshape Victor Kleptsyn}\\
	CNRS, Institut de R\'echerche Math\'ematique de Rennes (IRMAR, UMR 6625)\\
	B\^at. 22-23, Campus Beaulieu,
	263 avenue du G\'en\'eral Leclerc,
	35042 Rennes, France\\
	email: \texttt{victor.kleptsyn@univ-rennes1.fr}
	
	\smallskip

	{\scshape Dominique Malicet}\\
	Laboratoire d'Analyse et de Math\'ematiques Appliquées (LAMA, UMR 8050)\\
	Universit\'e Gustave Eiffel\\
	5  bd.~Descartes, 77454  Champs  sur Marne, France\\
	email: \texttt{dominique.malicet@crans.org}
	
	\smallskip
	
	{\scshape Carlos Meni\~no Cot\'on}\\
	CITMAGA \& Departamento de Matemática Aplicada I\\
	Instituto Investigaciones Tecnológicas\\
	Rúa Constantino Candeira, CP15782, Santiago de Compostela, Spain\\[.3em]
	Escola de Enxeñería Industrial\\
	Universidade de Vigo\\
	Rúa Conde de Torrecedeira 86, CP36208, Vigo, Spain.\\
%
	email:  \texttt{carlos.menino@uvigo.gal}

	\smallskip
	
	{\scshape Michele Triestino}\\
	Institut de Math\'ematiques de Bourgogne (IMB, UMR 5584)\\
	Universit\'e de Bourgogne\\
	9 av.~Alain Savary, 21000 Dijon, France\\
	email: \texttt{michele.triestino@u-bourgogne.fr}
	
\end{flushleft}

\end{document}